\documentclass[12pt]{amsart}
\usepackage{amsfonts}
\usepackage{amssymb}
\usepackage{amsmath}
\usepackage{amsthm}
\usepackage{enumerate}
\usepackage[all]{xy}
\usepackage{graphicx}
\usepackage{appendix}
\newtheorem{theorem}{Theorem}
\newtheorem{corollary}{Corollary}

\newtheorem{remark}{Remark}

\newtheorem{example}{Example}
\newtheorem{proposition}{Proposition}

\theoremstyle{definition}
\newtheorem{definition}{Definition}

\def\Qp{\mathbb{Q}_{p}}
\def\Zp{\mathbb{Z}_{p}}

\def\P{\mathbb{P}^{1}(\mathbb{Q}_p)}

\def\Q{\mathbb{Q}}

\begin{document}

\title[]{Oscillating Sequences, Minimal Mean Attractability and Minimal Mean-Lyapunov-Stability}

\author{Aihua Fan and Yunping Jiang}


\address{Fan Ai-Hua: LAMFA UMR 7352, CNRS \\
Facult\'e des Sciences \\
Universit\'e de Picardie Jules Verne \\
33, rue Saint Leu \\
80039 Amiens CEDEX 1, France}
\email[]{ai-hua.fan@u-picardie.fr}

\address{Yunping Jiang: Department of Mathematics\\
Queens College of the City University of New York\\
Flushing, NY 11367-1597\\
and\\
Department of Mathematics\\
Graduate School of the City University of New York\\
365 Fifth Avenue, New York, NY 10016}
\email{yunping.jiang@qc.cuny.edu}



\thanks{The first author  is partially supported by NSFC [grant number 11471132]. The second author is partially supported by the collaboration grant from the Simons Foundation [grant number 199837] and the CUNY collaborative incentive research grants [grant number 2013] and awards from PSC-CUNY and a grant from NSFC [grant number 11171121].}

\begin{abstract}
We define oscillating sequences which include the M\"obius function in the number theory.
We also define minimally mean attractable flows and  minimally mean-L-stable flows.
It is proved that all oscillating sequences are linearly disjoint from minimally mean attractable and
minimally mean-L-stable flows. In particular,  that is the case for the M\"obius function.
Several minimally mean attractable and
minimally mean-L-stable flows are examined. 
These flows include the ones defined by all $p$-adic polynomials,  all $p$-adic rational maps with good reduction, all automorphisms of $2$-torus with 
zero topological entropy, all diagonalized affine maps of $2$-torus with zero topological entropy, all Feigenbaum zero topological entropy flows, and all  orientation-preserving  circle homeomorphisms.
\end{abstract}

\maketitle

\section{\bf Introduction}
Consider a pair $\mathcal{X} = (X,T)$, where $X$ is a compact metric space
with metric $d (\cdot, \cdot)$ and $T: X\to X$ is a continuous map. We call $\mathcal{X}$ a {\em flow} or a {\em dynamical system} because we will consider iterations $\{ T^{n}\}_{n=0}^{\infty}$.
A sequence of complex numbers $\xi(n)$, $n=1, 2,\cdots$,  is said to be {\em observed} in $\mathcal{X}$ if there is a continuous function
$f: X\to \mathbb{C}$ and a point $x \in X$ such that $\xi(n) = f(T^n x)$.

Given a sequence of complex numbers ${\bf c}: =(c_{n})$, $n=1, 2, \cdots$.
We say that it is {\em linearly disjoint} from the flow $\mathcal{X}$ if we have
\begin{equation}\label{disjointness}
   \lim_{N\to \infty} \frac{1}{N}\sum_{n=1}^N c_n \xi(n) =0
\end{equation}
for any observable $\xi(n)$ in $\mathcal{X}$.
An interesting sequence is the M\"obius function $c_{n}:=\mu (n)$. 
Recall that by definition $\mu(1) =1$, $\mu(n)=(-1)^r$ if $n$ is a product of $r$ distinct primes, and $\mu(n)=0$
if  $n$ is not square-free. Sarnak has conjectured that the M\"{o}bius function $\mu(n)$ is linearly disjoint from
all flows with zero topological entropy (see~\cite{Sa1,Sa2}).
This is now called Sarnak's conjecture, which remains open in its generality although it is proved in several
special cases~\cite{Bou,BSZ,GT,Kar,LS,MR}.  
As pointed out in~\cite{Sa1,Sa2}, this conjecture has
a connection with many important problems in number theory, for example, the Riemann hypothesis.
More recently, in~\cite{Ge}, Ge studied 
Sarnak's conjecture in association with a $C^*$ algebra and its maximal
ideal space and in~\cite{EKLR},  which contains a good survey of the subject and a rather complete list of refereces, El. Abdalaoui et al.  considered 
the Sarnak conjecture for sequences of numbers in $\{-1,0,1\}$ by comparing  several natural generalizations. 

An important property satisfied by the M\"obius function is
\begin{equation}~\label{pnt}
\sum_{n=1}^{N} \mu (n) =o(N)
\end{equation}
which is equivalent to the prime number theorem that the number of prime numbers less than or equal to $N$ is approximately $N/\ln N$.
A stronger statement (conjecture) is that for $\epsilon >0$
\begin{equation}~\label{pnt}
\sum_{n=1}^{N} \mu (n) =O_{\epsilon}(N^{\frac{1}{2}+\epsilon})
\end{equation}
which is equivalent to the Riemann hypothesis.
Another property satisfied by the M\"obius function is Davenport's theorem~\cite{Da} that
\begin{equation}\label{daven}
\sum_{n=1}^{N} \mu (n) e^{2\pi i \alpha n} = o(N)
\end{equation}
for any real number $0\leq \alpha <1$
and Hua's theorem~\cite{Hua}
\begin{equation}\label{hua}
\sum_{n=1}^{N} \mu (n) e^{2\pi i \alpha n^{k}} = o(N)
\end{equation}
for any real number $0\leq \alpha <1$ and any integer $k\geq  2$.
The reader can refer to~\cite{LZ} for more delicate estimates for these sums for $k\geq 1$.

In this paper,  we will not only consider the M\"obius function but also
a sequence of complex numbers ${\bf c}=(c_n)$ satisfying the oscillating condition
\begin{equation}\label{oscillating}
\sum_{n=1}^{N} c_n e^{-2\pi i n t} =o_{t} (N), \quad \forall t \in [0, 1),
\end{equation}
 and the growth condition
\begin{equation}\label{growth}
   \sum_{n=1}^{N} |c_n|^{\lambda} =O(N)
\end{equation}
for some $\lambda>1$
(such a sequence ${\bf c}$ will be called an oscillating sequence).
There are many arithmetic functions, including
the M\"obius function and the Liouville function, which are oscillating sequences  (see~\cite{DD,De}).
Recall that the Liouville function $l(n)$ is defined as $l(n) =(-1)^{\Omega (n)}$ where $\Omega(n)$ is 
the number of prime factors of $n$, counted with multiplicity. 
On the other hand,  we will define a minimally mean attractable flow and a minimally mean-L-stable flow.
We will prove that all oscillating sequences are linearly disjoint from all minimally mean attractable and minimally
mean-L-stable flows. We will prove that $p$-adic polynomial flows, $p$-adic rational flows with good reduction,  automorphism flows 
with zero topological entropy on $2$-torus, all diagonalized affine flows of $2$-torus with zero topological entropy, all Feigenbaum zero topological 
entropy flows and all orientation-preserving circle homeomorphism flows are all 
minimally mean attractable and minimally mean-L-stable flows. 

We organize our paper as follows. In Section 2, we define oscillating sequences (Definition~\ref{os}). In Section 3, we define minimally mean-L-stable flows (Definition~\ref{mmls}) and minimal mean attractable flows (Definition~\ref{mma}).
We will also review mean-L-stable flows and
its equivalent notion, mean-equicontinuous flows.
In Section 4, we prove that all oscillating sequences are linearly disjoint from all  minimal mean attractable and minimally mean-L-stable flows (Theorem~\ref{main1}).
In particular, as an example, we have that the M\"obius function is linearly disjoint from all  minimal mean attractable and minimally mean-L-stable flows (Corollary~\ref{scmmammls}).
As another example,  we have that all oscillating sequences
are linearly disjoint from all equicontinuous flows (Corollary~\ref{M2}).
In Section 5, we discuss $p$-adic polynomial flows and $p$-adic rational flows.
All polynomials with $p$-adic integral coefficients define equicontinuous flows on the ring of $p$-adic integers.
Thus they are minimal mean attractable and minimally mean-L-stable flows. 
All $p$-adic rational maps with good reduction define minimal mean attractable and minimally mean-L-stable flows on the projective line $\mathbb{P}^1(\mathbb{Q}_p)$.
Therefore,  all oscillating sequences are linearly disjoint from all these $p$-adic polynomial flows and  $p$-adic rational flows
(Corollary~\ref{pp} and Corollary~\ref{rp}). 
In Section 6, we explain how to use our method to study affine maps and automorphisms of $2$-torus with zero topological entropy.
We first prove that in the diagonalized case, all oscillating sequences are linearly disjoint from all flows defined by affine maps of $2$-torus with zero topological entropy (Corollary~\ref{diagcoro})
since these flows are equicontinuous (Proposition~\ref{diag}).  
We then prove that all flows defined by automorphisms of $2$-torus with zero topological entropy are minimal mean attractable 
and minimally mean-L-stable flows~\label{nondiagthm} (they are not equicontinuous).
Thus all oscillating sequences are linearly disjoint from all flows defined by automorphisms of $2$-torus with zero topological entropy (Corollary~\ref{nondiagcoro}). Furthermore, we give  an affine map of $2$-torus with zero topological entropy and  an oscillating sequence such that this oscillating sequence is not linearly disjoint from the flow defined by this affine map. Notice that Liu and Sarnak has proved that the M\"obius function is linearly disjoint from this example in~\cite{LS}. We further note that in order to be linearly disjoint from all flows with zero topological entropy, an oscillating sequence should not only be oscillating in the first order but also be oscillating for any higher order (Remark~\ref{higher}) and the M\"obius function has this property as shown by Hua in~\cite{Hua} (see Equation (\ref{hua})). 
In Section 7, we prove that all Feigenbaum zero topological entropy flows are minimal mean attractable and minimally mean-L-stable flows. 
Thus all oscillating sequences are linearly disjoint from all Feigenbaum zero topological entropy flows (Theorem~\ref{main2}). 
In Section 8, we prove that Denjoy counter-examples are minimal mean attractable but not equicontinuous even 
when restricted to their non-wandering sets, which are minimal subsets. However, we  prove that Denjoy counter-examples 
are minimally mean-L-stable (Theorem~\ref{ceq}). In the proof of this theorem, we first prove that the flow defined by  
a Denjoy counter-example is minimally mean-L-stable when its non-wandering set   
has zero Lebesgue measure,  and then we show that every Denjoy counter-example is conjugate to a 
Denjoy counter-example whose non-wandering set has zero Lebesgue measure (an obseravtion pointed out to us by Davit Karagulyan).   
Thus all oscillating sequences are linearly disjoint from all flows defined by Denjoy counter-examples (Corollary~\ref{scdenjoy}).
The case of M\"obius sequence was already considered by Karagulyan \cite{Kar} and now is one of the consequences of Theorems~\ref{main1} and~\ref{ceq} (Corollary~\ref{Kar}).

\medskip
\medskip
\noindent {\em Acknowledgement.} We would like to thank Professors Hedi Daboussi, Liming Ge,  Davit Karagulyan,  Jianya Liu, J\"{o}rg Schmeling,  Jie Wu, Xiangdong Ye and Enrique Pujals  for sharing information and having many interesting discussions with us. We would also like to thank Professor Peter Sarnak for his encouragement comments on our initial version of this paper. 
The work was started when the first  author was visiting Lund University and he would like to thank the Knut and Alice 
Wallenberg Foundation for its support.  Both authors would like to thank 
the Academy of Mathematics and Systems Science at the Chinese Academy of Sciences for its hospitality when they visited there and worked on this project.

 \section{\bf Oscillating Sequences}

 In classical analysis (see~\cite{Kahane}), a sequence of complex numbers $(u_n)_{n\in \mathbb{Z}}$ is called a generalized almost
periodic sequence if for any $t\in [0, 1)$ the following limit exists
$$
c(t) =\lim_{N\to \infty} \frac{1}{2N+1}\sum_{n=-N}^N u_n e^{- 2\pi i n t}.
$$
Almost periodic and generalized almost periodic sequence were studied by Bohr, Hartman, et al.
 The spectrum of a generalized almost periodic sequence is the set of $t\in[0, 1)$ such that $c(t)\not=0$.  In general, the spectrum is countable.
We are interested in generalized almost periodic sequence with empty spectrum defined on positive integers. Thus we introduce the following definition.

\medskip
\begin{definition}~\label{os}
Let ${\bf c}:=(c_n)$, $n=1, 2, \cdots$, be a sequence of complex numbers.
We say that ${\bf c}$ is an {\em oscillating sequence} if
the Cesaro means
$$
\sigma_N(t):= \frac{1}{N}\sum_{n=1}^{N} c_n e^{-2\pi i n t}
$$
converge to zero as $N$ tends to the infinity for every $0\leq t <1$.
\end{definition}

\medskip
\begin{remark}
All oscillating sequences in this paper are also assumed satisfying the growth condition (\ref{growth}).
\end{remark}

\medskip
If we only assume that $\lim \sigma_N(\alpha)$ exists but is not necessarily zero for every
$0\le t <1$, 
we denote
$$
 \mathcal{Z}({\bf c}) = \left\{t \in [0,1)\;|\;  \lim_{N\to \infty} \sigma_N(t)=0\right\}.
$$
So, the statement that ${\bf c}$ is an oscillating sequence means
$
\mathcal{Z}({\bf c})=[0, 1).
$
The complementary of $\mathcal{Z}({\bf c})$ is defined to be the spectrum of {\bf c}.

Let us give several examples of oscillating sequences.

\medskip
\begin{example}~\label{ex1}
The sequence of complex numbers
$
{\bf c} =\big( e^{ 2\pi i n \alpha}\big)
$
for a fixed $0\leq \alpha <1$ is 
not an oscillating sequence, because
$
\mathcal{Z}({\bf c})= [0, 1)\setminus \{ \alpha\}.
$
Note that the spectrum is one-point set $\{\alpha\}$.
\end{example}

\medskip
\begin{example}~\label{ex2}
The sequence of complex numbers
$
{\bf c}= \big( e^{2\pi i c n \log n}\big),  \;\; c >0,
$
is an oscillating sequence because the Cesaro means
$$
\sigma_N(t)= \frac{1}{N}\sum_{n=1}^N e^{2\pi i c n \log n} e^{-2\pi i n t} = O\Big(\frac{1}{\sqrt{N}}\Big)
$$
uniformly on $0\leq t <1$ (\cite[Vol. 1, p. 197]{Z}).
\end{example}

\medskip
\begin{example}~\label{ex3}
The sequence of complex numbers
$
{\bf c} = \big( e^{2\pi i n^2 \alpha}\big)
$
for any fixed irrational $\alpha$ is an oscillating sequence
because
the sequence
$$
\{n^2\alpha - n t   \pmod{1}   \}_{n=0}^{\infty} 
$$
is uniformly distributed on the unit interval $[0,1]$ for any fixed $0\leq t <1$.
So the Cesaro means
$$
\sigma_N(t)= \frac{1}{N}\sum_{n=1}^N e^{2\pi i  (n^2\alpha - n t)}=o_{t} (1).
$$
\end{example}

\medskip
\begin{proposition}\label{n2}
 Consider the sequence of complex numbers
$
{\bf c} : = \big( e^{2\pi i n^2 \alpha}\big)
$
where  $0\leq \alpha =p/q<1$ is a rational number with $(p, q)=1$.   The spectrum
 of ${\bf c}$ is the set of rational numbers $\frac{r}{s}$  such that
 $s|q$ and
 $$
 \sum_{k=0}^{q-1} e^{2\pi i (k^2p/q + k r/s)}\not=0.
 $$
\end{proposition}

\begin{proof}
We have
$$
\lim_{N\to \infty} \frac{1}{N}  \sum_{n=1}^N e^{2\pi i n^2 \alpha} e^{-2\pi i n t}
   =\sum_{k=0}^{q-1} e^{2\pi i (k^2p/q + k t)} \lim_{N\to \infty} \frac{1}{N}   \sum_{m: mq + k \le N}
    e^{2\pi i mq t}.
$$
If $t$ is irrational or if $t=r/s$ with $s\not|q$, the last sum is bounded and the above limit is zero.
Then such a $t$ is not in the spectrum. If $s|q$, we have
$$
\lim_{N\to \infty} \frac{1}{N}  \sum_{n=1}^N e^{2\pi i n^2 \alpha} e^{-2\pi i n r/s}
   =\frac{1}{q}\sum_{k=0}^{q-1} e^{2\pi i (k^2p/q + k r/s)}.
$$
\end{proof}

It is easy to check that
$$
   \mathcal{Z}(\{ e^{2\pi i n^2 \cdot \frac{1}{2}} \}) =[0, 1)\setminus \big\{ \frac{1}{2}\big\},
$$
$$
    \mathcal{Z}(\{ e^{2\pi i n^2 \cdot \frac{1}{3}} \})= \mathcal{Z}(\{ e^{2\pi i n^2 \cdot \frac{2}{3}} \}) =[0, 1)\setminus \big\{ 0, \frac{1}{3}, \frac{2}{3}\big\},
$$
$$
    \mathcal{Z}(\{ e^{2\pi i n^2 \cdot \frac{1}{4}} \}) = \mathcal{Z}(\{ e^{2\pi i n^2 \cdot \frac{3}{4}} \}) =[0, 1)\setminus \big\{0,  \frac{1}{2}\big\}.
$$

The following example is due to Davenport~\cite{Da} (see also~\cite{GT2}).

\medskip
\begin{example}~\label{ex4}
The M\"obius function $\big(\mu (n)\big)$ is an oscillating sequence.
Actually a stronger result holds for the Cesaro means
$$
\sigma_N (t) =\frac{1}{N} \sum_{n=1}^{N} \mu(n) e^{- 2\pi i n t} = O\Big(\frac{1}{\log^h N}\Big)
$$
for any $h>0$ where the estimate is uniform on $0\leq t <1$.
\end{example}

Another example can be obtained from the following proposition.
Let $(\xi_n)$ be a sequence of independent and identically
distributed real random variables such that 
$$
\mathbb{E} e^{\lambda \xi_1}\le e^{\frac{\lambda^2}{2}}, \quad \forall \lambda \in \mathbb{R}.
$$
Such a sequence is called a subnormal sequence.

\medskip
\begin{proposition}~\label{ex5}
Suppose $(\xi_{n}$) is a subnormal sequence.  Let $(u_n)$ be a sequence of real numbers such that
$u_n = O(n^\tau)$ for some $0<\tau<1/2$. Then almost surely the random sequence
${\bf c}: = (u_n \xi_n)$ is an oscillating sequence.
\end{proposition}

\begin{proof}
Actually, as a consequence of the Salem-Zygmund inequality on the uniform estimate
of random trigonometric polynomials and as we refer to~\cite[p.73]{Kahane}, where
the Rademacher random sequence was considered but the proof for subnormal sequences is the same,
we have almost surely the Cesaro means
$$
\sigma_N (t) =\frac{1}{N} \sum_{n=1}^{N} u_n \xi_n e^{-2\pi i n t} = O\Big(\frac{\sqrt{\log N}}{N^{\frac{1}{2}-\tau}}\Big)$$
where the estimate is uniform on $0\leq t <1$.
This proves the proposition.
\end{proof}

From \cite{FS} we can get other  oscillating sequences similar to those in Proposition~\ref{ex5}.
The next example is due to Daboussi-Delange \cite[p. 254]{DD}.

\medskip
\begin{proposition}~\label{ex6}
Let ${\bf c}=(c_n)$ be a multiplicative arithmetical function such that
$|c_n|\le 1$ for all $n$.  Then ${\bf c}$ is an oscillating sequence if and only if
for any Dirichlet character $\chi$ and and real number $u$, we have
$$
\sum_{p }
    \frac{1}{p} \Big( 1 - \mbox{\rm Re} \ \big( \chi(p) f(p) p^{-i u} \big) \Big) =\infty
$$
where the sum is taken over prime numbers $p$.
\end{proposition}

This classical result of Daboussi and Delange is generalized to
 asymptotic
orthogonality of multiplicative functions to “irrational
” nilsequences by N. Frantzikinakis and B. Host \cite{FH2014}.

The oscillating sequences shares the following property that arithmetic subsequences of an oscillating sequence are oscillating.

\begin{proposition} Suppose that ${\bf c}= (c_{n})$ is an oscillating
	sequence. Then for any $q\ge 2$ and any $r=1, 2, \cdots, q$
	we have 
	\begin{equation}\label{c-mod-q}
	\lim_{N\to \infty} \frac{1}{N} \sum_{\stackrel{1\le n\le N}{n\equiv r\, (\!\!\!\!\!\!\mod q)}}
	c_{n}e^{-2\pi i t n} =0, \quad \forall t \in [0, 1).
	\end{equation}
\end{proposition}

\begin{proof} The idea comes from \cite{DD}. The proof is based on the fact that 
	$$
	\frac{1}{q}\sum_{j=1}^q   e^{\frac{2\pi i j a}{q}} = 1 \ \mbox{\rm or}\ 0 
	$$
	according to $a \equiv 0$ or $\not\equiv 0$ ($\!\!\!\!\mod q$). It follows that
	\begin{eqnarray*}
		\frac{1}{N} \sum_{\stackrel{1\le n\le N}{n\equiv r\, (\!\!\!\!\!\!\mod q)}}
		c_{n} e^{-2\pi i t n} 
		&=& 
		\frac{1}{N}  \sum_{1\le n\le N}  c_{n} e^{-2\pi i t n} \frac{1}{q}\sum_{j=1}^q   e^{\frac{-2\pi i j (n-r)}{q}}\\
		&=&
		\frac{1}{q}\sum_{j=1}^q   e^{\frac{2\pi i jr}{q}}  \frac{1}{N}  \sum_{1\le n\le N}  c_{n} e^{-2\pi i \big(t+\frac{j}{q}\big) n}
		\to 0
	\end{eqnarray*}
as $N\to \infty$, since ${\bf c}$ is an oscillating sequence. 
	
\end{proof}

The proof shows that for (\ref{c-mod-q}) to hold for a fixed $t$ it suffices that
$t+ j/q\in \mathcal{Z}({\bf c})$ for all $j=1, 2, \cdots, q$. 
\medskip

 In \cite{KS}, Kahane and Saias studied completely multiplicative functions with zero sum.  These functions share another kind of oscillating property.

\section{\bf Minimal Mean Attractability and Minimal Mean-L-Stabity}
In this section we define a class of flows from which all oscillating sequences will be proved to
be linearly disjoint.
We start with recalling the classical definition of equicontinuous flows.

\medskip
\begin{definition}~\label{eqcn}
          A flow $\mathcal{X}=(X,T)$ is said to be {\em equicontinuous}
if the family $\{T^n\}_{n=0}^{\infty}$ is equicontinuous. That is to say, for every $\epsilon > 0$
there is a $\delta > 0$ such that whenever $x, y \in X$ with $d(x, y) <\delta$,
we have $d(T^n x,T^n y) <\epsilon$ for $n = 0,1,2, \cdots$.
\end{definition}

It is well known that a flow $\mathcal{X}=(X,T)$ with $T$ being surjective is equicontinuous if and only if there exists
a compatible metric $\rho$ on $X$ such that $T$ acts on $X$ as an isometry, i.e., $\rho(T x,Ty)=\rho (x, y)$ for all $ x, y \in X$.
Thus, if $T$ is surjective and equicontinuous, then $T$ must be a homeomorphism.
Moreover, it is also known that for an equicontinuous flow $\mathcal{X}$, if $T$ is a homeomorphism,
then $T$ acts as a minimal system on the closure $\overline{O(x)}$ of the forward orbit $O(x)=\{T^{n}\}_{n=0}^{\infty}$ of any $x\in X$  (see ~\cite{Pe}). Thus, for an equicontinuous flow $\mathcal{X}$, when $T$ is a transitive homeomorphism, it is always conjugate
to a minimal rotation on a compact abelian metric group. In this case, let $m$ be the unique Haar probability measure on $X$, then the measurable dynamical system $(X,T,m)$
has discrete spectrum.
An equicontinuous flow $\mathcal{X}$ such that $T$ is a homeomorphism can be  decomposed
into  minimal subsystems. Sarnak's conjecture holds for such a flow.  
Actually, we will prove that Sarnak's conjecture holds  for a  much larger class of flows which we will call minimally mean attractable and minimally mean-L-stable flows. We give a more detailed description below.

We first weaken the equicontinuity condition following Fomin in~\cite{Fomin}. Suppose $\mathbb{N}=\{ 1,2, \cdots, \}$ is the set of natural numbers.
Let $E$ be a subset of $\mathbb{N}$. The upper density of $E$ is, by definition,
$$
\overline{D}(E) = \limsup_{n\to\infty} \frac{\sharp (E\cap[1, n])}{n}.
$$

\medskip
\begin{definition}~\label{mma}
A flow $\mathcal{X}=(X,T)$ is said to be {\em mean-L-stable} ({\bf MLS} for short) (here $L$ recalls the sense of Lyapunov)
if for every $\epsilon > 0$, there is a $\delta > 0$ such that
$d(x, y) <\delta$ implies $d(T^nx,T^ny) <\epsilon$ for all $n=0, 1, 2, \cdots$ except
a subset of natural numbers with the upper density less than $\epsilon$.
\end{definition}

Fomin proved that a minimal and {\bf MLS} flow is uniquely ergodic in~\cite{Fomin}.
Oxtoby proved that a transitive and {\bf MLS} flow is uniquely ergodic in \cite{Oxtoby1952}.
Li, Tu and Ye further proved that every ergodic invariant measure on a {\bf MLS} flow has discrete spectrum.

Any flow $\mathcal{X}$ admits a minimal sub-flow, for example, the restriction on the $\omega$-limit set of any point.

\medskip
\begin{definition}~\label{mmls}
We say that a flow $\mathcal{X}=(X, T)$ is {\em minimally {\bf MLS}} ({\bf MMLS} for short) if
for every minimal subset $K\subseteq X$, the sub-flow $\mathcal{K}=(K, T)$ is {\bf MLS}.
\end{definition}

Following \cite{LTY}, we say the flow ${\mathcal X}$ is {\em mean-equicontinuous} at a point $x \in X$ if for every $\epsilon > 0$, there is
$\delta > 0$ such that for every $y \in X$ with $d(y, x)<\delta$ we have
\begin{equation}\label{MEC}
    \limsup_{n\to\infty} \frac{1}{n} \sum_{k=0}^{n-1} d(T^kx, T^k y) <\epsilon.
\end{equation}
We say that the flow ${\mathcal X}$ is mean-equicontinuous if it is mean-equicontinuous at every point $x\in X$. That is, by the compactness of $X$,
the above inequality (\ref{MEC}) holds for all $x, y\in X$ such that $d(y, x)<\delta$. The following proposition is proved by Li, Tu, and Ye in~\cite{LTY}

\medskip
\begin{proposition}~\label{eq}
A flow ${\mathcal X}$ is mean-equicontinuous if and only if it is {\bf MLS}.
\end{proposition}

Let $C(X)$ be the space of all continuous functions $f: X\to \mathbb{C}$ with maximum norm
$$
\|f\|_{\infty} =\max_{x\in X} |f(x)|.
$$

\medskip
\begin{proposition}~\label{meaneq}
Suppose the flow ${\mathcal X}$ is {\bf MLS}. Suppose
${\bf c}=(c_n)$ is a sequence satisfying the growth condition (\ref{growth}).
Then for any continuous function $f\in C(X)$,
$$
 S_Nf(x) =\frac{1}{N} \sum_{n=1}^{N} c_n f(T^n x), \quad (N=1, 2, \cdots)
$$
is an equicontinuous sequence in $C(X)$.
\end{proposition}

\begin{proof}
Suppose $\|f\|_{\infty}\not= 0$, otherwise it is trivial. 
Take a $\lambda > 1$ satisfying (\ref{growth}). Suppose $\lambda'> 1$ be the  number such that
$
\frac{1}{\lambda} + \frac{1}{\lambda'}=1.
$
There is a positive constant $C$ (we can take $C$ to be the supremum of $(\frac{1}{N} \sum_{n=1}^{N} |c_{n}|^{\lambda} )^{1/\lambda}$) such that
for any $x, y\in X$, we have
 \begin{equation}~\label{Holder1}
 |S_Nf(x) - S_N f(y)| \leq C  \Big(\frac{1}{N} \sum_{n=0}^{N-1} |f(T^n x) - f(T^n y)|^{\lambda'}\Big)^{1/\lambda'}.
 \end{equation}
This is a direct consequence of the H\"{o}lder inequality and (\ref{growth}).

By the uniform continuity of $f$, for any $\epsilon >0$ there exists $\eta >0$ such that
\begin{equation}\label{uniform_con}
d(x, y)< \eta \Rightarrow |f(x)- f(y)|<\frac{\epsilon}{2 C}.
\end{equation}
 We can take $\eta \le \frac{\epsilon}{4C\|f\|_\infty}$.
 Since $\mathcal{X}$ is {\bf MLS}, there is a $\delta' > 0$ such that $x, y\in X$ with
$d(x, y) <\delta'$ implies $d(T^nx,T^ny) <\eta$ for all positive integer $n$ except a set, denoted $E$, of upper density less than $\eta$.
There exists an integer $N^*$ such that
$$
     \frac{\sharp (E \cap [1, N])}{N} < \eta, \qquad \forall N \ge N^*
$$
Thus by (\ref{Holder1}), if $N\ge N^*$ and if $ d(x, y)<\delta'$ we have
\begin{equation}\label{difference*}
    |S_Nf(x) - S_N f(y)| < C \big(2 \|f\|_\infty \eta + \frac{\epsilon}{2C}\big)= \epsilon.
\end{equation}
In the above inequality, we first split the sum in (\ref{Holder1}) into two sums according to
$d(T^nx, T^n y)\ge \eta$ or $<\eta$.

The finite many family $\{S_1f , \cdots, S_{N^*}f\}$ being equicontinuous, there exists $\delta''$
such that
\begin{equation}\label{difference**}
      d(x, y)<\delta'' \Rightarrow \max_{1\le N\le N^*} |S_Nf(x) -S_Nf(y)|<\epsilon.
\end{equation}
Finally we conclude the proposition from (\ref{difference*}) and (\ref{difference**}).
\end{proof}

In~\cite[p. 575]{Auslander}, Auslander had decomposed a {\bf MLS} flow
into a star closed decomposition. In this decomposition each component  contains
precisely one minimal set and all invariant measures concentrated on minimal sets.
This implies that from measure-theoretical point of view, the union of these minimal sets is an attractor:
all points outside the union of these minimal sets are attracted into minimal sets eventually.
We will develop this idea into a new concept called the minimal mean attractablity.

 \medskip
 \begin{definition}~\label{mma}
Suppose $\mathcal{X}=(X,T)$ is a flow. Suppose $K$ is a closed subset of $X$ and $T: K\to K$ is minimal.
We say $x\in X$ is {\em mean attracted} to $K$ if for any $\epsilon >0$ there is a point $z=z_{\epsilon, x}\in K$
(depending on $x$ and $\epsilon$) such that
\begin{equation}\label{MA}
    \limsup_{N\to\infty} \frac{1}{N} \sum_{n=1}^{N} d(T^n x, T^n z) <\epsilon.
\end{equation}
The {\em basin of attraction} of $K$, denoted $\hbox{\rm Basin}(K)$, is defined to be the set of all
points $x$ which are mean attracted to $K$. It is trivial that
$K \subset \hbox{\rm Basin}(K)$.
We call $\mathcal{X}$ {\em minimally mean attractable} ({\bf MMA} for short) if
\begin{equation}\label{Decomp}
   X= \bigcup_{K} \mbox{\rm Basin}(K)
\end{equation}
where $K$ varies among all minimal subsets of $X$.
\end{definition}

Recall that a point $x\in X$ is attracted to $K$ if
$$
\lim_{n\to \infty} d(T^n x, K)=0.
$$
In general, this can not imply to that $x$ is mean attracted to $K$.
However, if $K$ is a periodic cycle, it does. 
We have even the following proposition.

\medskip
\begin{proposition} Let $x\in X$. Suppose that $\omega(x) =\{x_0,x_1, \cdots, x_{k-1}\}$ is a periodic cycle of period $k\geq 1$.
Then for any $\epsilon >0$ sufficiently small, there exists an integer $\tau \ge 0$ and
$z\in \omega(x)$ such that
$$
d(T^{\tau+m}x, T^{\tau +m} z) <\epsilon, \quad \forall m\ge 0.
$$
\end{proposition}

\begin{proof} Let $K=\omega(x)$.
Take  any $\epsilon >0$ sufficient small so that
the balls $B(x_j, \epsilon)$, $j=0,1,\cdots, k-1$, are disjoint. 
For any $j$ and any $\tau \ge 0$, consider the subset of the natural numbers
$$
R_\tau(x_j) =\{n\ge \tau:   T^n x \in B(x_j, \epsilon)\}.
$$
For any fixed $0\leq j<k$, 
$$
\lim_{R_\tau(x_j) \ni n \to \infty} T^n x = x_j.
$$
By the continuity of $T$, we have
$$
 \lim_{R_\tau(x_j)+1 \ni n \to \infty} T^n x = x_{j+1 \!\!\!\!\pmod{k}}.
$$
This implies that we can find a $\tau_{j}\geq 1$ such that 
\begin{equation}\label{recurrence}
     R_\tau (x_j)+1 \subset R_\tau(x_{j+1}), \quad \forall \tau\geq \tau_{j}.
\end{equation}
Let $\tau_{\max}=\max_{0\leq j<k} \tau_{j}$. Take $\tau \geq \tau_{\max} \in R_{\tau_{\max}} (x_{0})$. Then   
$$
d(T^\tau x , x_0) <\epsilon
$$ 
and from (\ref{recurrence}),
$$
d(T^{m+\tau}x, T^m x_0) = d(T^{m+\tau}x, x_{m \!\!\!\! \pmod{k}})<\epsilon, \quad \forall m\geq 0.
$$
We take $z\in K$ such that $T^\tau z =x_0$, then we have that
$$
d(T^{\tau+m}x, T^{\tau+m} z) = d(T^{\tau+m}x, x_{m \pmod{k}})<\epsilon, \quad \forall m\geq 0.
$$

\end{proof}

\medskip
\begin{example}~\label{1ex}
The flow
$$
([-1, 1], T(x)=-x^{2})
$$  
is {\bf MMA} and {\bf MMLS}.
\end{example}

\begin{proof}
The point $-1$ and $0$ are only two fixed points. Since $T'(-1)=-2$ and $T'(0)=0$, $-1$ is the repelling fixed point and $0$ is the attracting fixed point.
We have that
$$
    \mbox{\rm Basin}(\{-1\}) = \{-1, 1\}, \quad \hbox{and}\quad
    \mbox{\rm Basin}(\{0\}) = (-1, 1).
$$
Thus we have  
$$
[-1,1]= {\rm Basin}(\{-1\})\cup {\rm Basin}(\{0\}).
$$ 
This implies that $\mathcal{X}$ is {\bf MMA} and {\bf MMLS}.
\end{proof}

This flow is actually equicontinuous on any closed interval $[-a,a]\subset (-1,1)$. The next example is different (see Remark~\ref{neq}).

Let $\mathbb{T}^{2}=\mathbb{R}^{2}/\mathbb{Z}^{2}$ be the standard 2-torus. Consider the automorphism
$$
T_{t}(x,y) = (x+ty,y): \mathbb{T}^{2}\to \mathbb{T}^{2}
$$
for a fixed integer $t$.

\medskip
\begin{example}~\label{2ex}
The flow 
$$
\mathcal{X}= (\mathbb{T}^{2}, T_{t})
$$
is {\bf MMA} and {\bf MMLS}. Actually we have that all
automorphisms on $\mathbb{T}^2$ with zero topological entropy are {\bf MMA} and {\bf MMLS}.
\end{example}

\begin{proof}
Let $\mathbb{T}=\mathbb{R}/\mathbb{Z}$ be the circle.  Every circle $\mathbb{T}_{y}=\mathbb{T}\times \{y\}$ ($y$ being fixed) is $T_{t}$-invariant.
The restriction of $T_{t}$ on each $\mathbb{T}_{y}$ is a one-dimensional rotation with the rotation number $ty$. If $ty$ is irrational,
then $T_{t}: \mathbb{T}_{y}\to \mathbb{T}_{y}$ is minimal.
If $ty$ is rational, then every point $(x,y)\in \mathbb{T}_{y}$ is a periodic point. Then the restriction of $T_{t}$ on each of these periodic orbits is minimal.
In this case, $\mathbb{T}_{y} = \cup_{(x,y)\in \mathbb{T}_{y}} Orb ((x,y))$. Furthermore, we have that the minimal decompsition
$$
\mathbb{T}^{2}=\Big( \cup_{ty: irrational} \mathbb{T}_{y}\Big)  \cup \Big( \cup_{ty: rational}\cup_{(x,y)\in \mathbb{T}_{y}} Orb\big( (x,y)\big)\Big) .
$$
So  this flow is {\bf MMA}.

Since $T_{t}$ on each $\mathbb{T}_{y}$ is isometric, so it is equicontinuous on each of its minimal subset. So it is {\bf MMLS}.

As we shall see in Section 6,  any automorphism $T$ on $\mathbb{T}^2$ with zero topological entropy is (linearly)
topologically conjugate to a $T_{t}$ for some integer $t$. So the flow $\mathcal{X}=(\mathbb{T}^{2}, T)$ is also {\bf MMA} and {\bf MMLS}. 
\end{proof}

\medskip
\begin{remark}~\label{neq}
The flow $(\mathbb{T}^{2}, T_{t})$ is not mean-equicontinuous (equivalently, not {\bf MLS})
because the rotation number of $T_{t}|_{\mathbb{T}_{y}}$ varies with $y$.  
\end{remark}
 
 \section{\bf Disjointness}

Recall that $\mathcal{X}=(X,T)$ is a flow if $X$ is a compact metric space with metric $d$ and $T:X\to X$
is a continuous map.  A sequence $(\xi(n))$ is said to be {\em observable} in $\mathcal{X}$
if there exists a continuous function $f\in C(X)$ and a point $x\in X$ such that $\xi(n) = f(T^n x)$ for all $n\ge 1$.
Following the idea of Sarnak~\cite{Sa1},  we say that a given sequence of complex numbers ${\bf c}=(c_n)$
is {\em linearly disjoint} from $\mathcal{X}$ if for any observable $(\xi(n))$ in $\mathcal{X}$, we have
\begin{equation}~\label{ld}
  \lim_{N \to \infty} \frac{1}{N}\sum_{n=1}^{N} c_n \xi(n)=0.
\end{equation}

The first main result in this paper is the following one.

\medskip
\begin{theorem}~\label{main1}
Any oscillating sequence ${\bf c}=(c_n)$ satisfying the growth condition (\ref{growth})
is linearly disjoint from all {\bf MMA} and {\bf MMLS} flows $\mathcal{X}=(X, T)$.
More precisely, 
for any continuous function $f\in C(X)$, we have that
$$
 \lim_{N \to \infty} \frac{1}{N}\sum_{n=1}^{N} c_n f(T^n x)=0
$$
for every $x\in X$ and the limit is uniform on each minimal subset.
\end{theorem}

\begin{proof}
By the hypothesis, $X$ has the decomposition in Equality (\ref{Decomp}).
So, we need only prove that
for any minimal subset $K\subset X$, we have
\begin{equation}\label{limsup-K}
      \forall x\in \mbox{\rm Basin}(K), \quad \lim_{N \to \infty} \frac{1}{N}\sum_{n=1}^{N} c_n f(T^n x)= 0.
\end{equation}
For this purpose, we define
$$
S_N f(x) =\frac{1}{N}\sum_{n=1}^{N} c_n f(T^{n}x), \quad N=1, 2, \cdots.
$$
We first prove that the sequence $\mathcal{S}= \{ S_{N} f(x)\}_{N=1}^{\infty}$ converges to $0$ uniformly on $K$.

The sequence $\mathcal{S}$ is uniformly bounded in $C(K)$ by $C \|f\|_{\infty}$ for some constant $C>0$ and is equicontinuous by Proposition~\ref{meaneq}.
So, it is precompact in $C(K)$ by the Arzela-Ascoli theorem. Therefore, every subsequence of $\mathcal{S}$ has a convergent subsequence.
What we have to prove is the claim that every convergent subsequence converges uniformly on $K$ to zero.

 Let us prove this claim by contradiction. Suppose that some subsequence $S_{N_j} f$
 converges uniformly to a continuous function $g \in C(K)$ which is not identically zero. Assume
 $|g(x_0)|>0$ for some $x_0\in K$ and we can actually assume that
 \begin{equation}\label{contradiction2}
 |g(x)|>0, \ \ \ \forall x\in B(x_0, r)=\{x\in K\;|\; d(x, x_{0}) <r\}
  \end{equation}
 for some $r>0$ by the continuity of $g$.
 On the other hand,
we consider the probability measures on $K$:
$$
\nu_{N} =\frac{1}{N}\sum_{n=1}^{N} \delta_{T^n x_0}
$$
where $\delta_y$ denote the Dirac measure concentrated at $y$.
Let $\nu$ be a weak limit of  a subsequence of the sequence $(\nu_{N})$. Then it is $T$-invariant.
Since $K$ is minimal,
 by the Gottschalk theorem (see~\cite[Lemma 4]{Gottschalk1944}), $x_0$ is almost periodic in the sense that for every $\epsilon >0$, the set of $n$ such that
 $d(x_0, T^n x_0)<\epsilon$ is relatively dense.
 This implies that $\nu(B(x_0, \epsilon))>0$ for any $\epsilon >0$ (the support of $\nu$ is actually whole minimal set $K$).

Now consider the measure-preserving dynamical system
$$
(K, \mathcal{B}(K), T, \nu)
$$
where $\mathcal{B}(K)$
is the Borel $\sigma$-field of $K$.
Let $\sigma_f$ be the spectral measure of $f$ with respect to the $T$-invariant probability measure $\nu$ on $K$, $f$ being considered as in $L^2(\nu)$ and $\sigma_f$ being defined on the unite circle
$$
S^{1}=\{ z\in {\mathbb C}\;|\; |z|=1\}.
$$
By the spectral lemma (see \cite[p. 94-95]{Krengel}), we have
$$
\|S_{N}f\|^2_{L^2(\nu)} =\int_{\mathbb T} \Big| \frac{1}{N}\sum_{n=1}^{N} c_n e^{2\pi i n t}\Big|^2 d\sigma_f(t).
$$
Since ${\bf c}=(c_{n})$ is an oscillating sequence, by the Lebesgue dominated convergence theorem, the right hand of the above equation converges to $0$.
This implies that $S_Nf$ converges to zero in $L^2(\nu)$-norm.
 In particular, $S_{N_j}f$ converges to zero in $L^2(\nu)$-norm. Consequently, there is a subsequence $N_j'$
 of $N_j$ such that $S_{N_j'}f$ converges to zero $\nu$-almost everywhere, which contradicts
(\ref{contradiction2}) and the fact that $\nu(B(x_0, r))>0$. The contradiction implies that the sequence $\mathcal{S}$ converges uniformly on $K$ to $0$.

We have just proved the limit (\ref{limsup-K})  for $x\in K$.
Now we are going to prove  (\ref{limsup-K}) for a general point $x \in  \mbox{\rm Basin}(K)$.
Let $\eta >0$ be an arbitrarily small number. By the uniform continuity of $f$, there is a $\epsilon >0$
such that $d(u, v)<\epsilon$ implies $|f(u) - f(v)| <\eta$. We assume that  $2\|f\|_\infty \epsilon \le \eta$. For $x\in \mbox{\rm Basin}(K)$,  by the definition of $\mbox{\rm Basin}(K)$, there exists $z=z_{\epsilon,x}\in K$
such that
   \begin{equation*}\label{MA2}
    \limsup_{N\to\infty} \frac{1}{N} \sum_{n=1}^{N} d(T^n x, T^n z) <\epsilon^2.
\end{equation*}
which implies that  $\overline{D}(E)\le \epsilon$
where $E=\{n\ge 0: d(T^n x, T^n z)\ge \epsilon\}$, because
$$
   \epsilon \sharp (E\cap [1, N])  \le \sum_{n=1}^{N} d(T^n x, T^n z).
$$
Write
$$
     S_Nf(x) = S_Nf(x) - S_Nf(z) +  S_Nf(z).
$$
As $z\in K$, we have proved that $S_Nf(z)$ tends to zero.
On the other hand, following the exact same proof of Proposition~\ref{meaneq}, we show that
\begin{equation*}\label{difference2}
    |S_Nf(x) - S_N f(z)| \leq C (2 \|f\|_\infty \epsilon + \eta)\le 2 C \eta.
\end{equation*}
Since $\eta>0$ is arbitrary, we have thus proved (\ref{limsup-K}) for every $x\in \mbox{\rm Basin}(K)$.
\end{proof}
 
A special case of Theorem~\ref{main1} is that
 
\medskip
\begin{corollary}~\label{scmmammls}
The M\"obius function is linearly disjoint from all {\bf MMA} and {\bf MMLS} flows.
\end{corollary}

\begin{proof}
This is because the  M\"obius function gives an oscillating sequence 
due to Davenport's theorem (see Example~\ref{ex4}).
\end{proof} 
  
Another consequence of Theorem~\ref{main1} is that   
  
\medskip
\begin{corollary}\label{M2}
Any oscillating sequence satisfying the growth condition (\ref{growth}) is linearly disjoint from all equicontinuous flows.
\end{corollary}

\begin{proof}
Let $\mathcal{X}=(X, T)$ be an equicontinuous flow.
The sequence of compact sets $\{ T^n X\}_{n=1}^{\infty}$ decrease to a non-empty compact set,
which we denote by $X_\infty$. The sub-flow $T: X_\infty \to X_\infty$ is surjective and equicontinuous and thus a homeomorphism.
The set $X_\infty$ is decomposed into minimal subsets.

We claim that each point $x\in X$ is mean attracted to some minimal set.  In fact, the sequence
$T^nx$ admits a subsequence $T^{n_j} x$ converging to a limit point $x_\infty$,  which must belong to
$X_\infty$. Recall the definition of equicontinuity for $T$: for any $\epsilon >0$ there exists $\eta >0$ such that
$$
    d(u, v)<\eta \Rightarrow d(T^n u, T^n v) <\frac{\epsilon}{2} \ \ \ (n=1, 2, \cdots).
$$
Take $j_0$ sufficiently large such that $d(T^{n_{j_0}} x, x_\infty) < \eta$. As $T: X_\infty \to X_\infty$
is a homeomorphism, there exist $z\in X_\infty$ such that $x_\infty = T^{n_{j_0}} z$ so
$$
d(T^{n_{j_0}} x, T^{n_{j_0}} z) < \eta.
$$
By the equicontinuity, we have
$$
 d(T^{n_{j_0} +n} x, T^{n_{j_0} +n} z) <\frac{\epsilon}{2}  \ \ \ (n=1, 2, \cdots).
$$
It follows that
$$
   \limsup_{N\to \infty} \frac{1}{N}\sum_{n=1}^{N} d(T^n x, T^n z)\le  \frac{\epsilon}{2} <   \epsilon.
$$
Thus we have proved that $x$ is mean attracted to
$$
K:=\overline{O(z)}=\overline{O(x_\infty)},
$$
where $O(\cdot)$ means the forward orbit. So we have proved that this flow is {\bf MMA}. But it is equicontinuous on each minimal subset, so it is also {\bf MMLS}.
\end{proof}

\medskip
\begin{remark}~\label{R1b}
As one can see in the proof of Theorem~\ref{main1}, 
we don't really need that the spectrum of the sequence ${\bf c}=(c_{n})$ is empty. What we really need is that the spectrum of ${\bf c}$ as a subset in the circle $S^1$ is disjoint from the support of $\sigma_f$ in the circle $S^1$.  In other words, the oscillating condition on ${\bf c}$ in Theorem~\ref{main1} can be relaxed  to that 
$$
\mathcal{Z}({\bf c}) \subseteq \hbox{supp}(\sigma_f) .
$$
\end{remark} 

As a direct consequence of the proof of Theorem~\ref{main1} and Remark~\ref{R1b}, we have that 

\medskip
\begin{corollary}~\label{M2}
	Suppose $S^{1} =\{ z\in \mathbb{C}\;|\; |z|=1\}$ is the unit circle. Let $R_\alpha (z) =e^{2\pi i \alpha} z$ 
	be the rigid rotation where $\alpha \in (0, 1)$ is irrational. 
	Any sequence ${\bf c}=(c_{n})$ satisfying the growth condition (\ref{growth}) 
	such that $\alpha \mathbb{Z}\subset \mathcal{Z}({\bf c})$
	is linearly disjoint from the flow $\mathcal{X}=(S^{1}, R_\alpha)$. 
\end{corollary}

\begin{proof} We use the notation in the proof of Theorem~\ref{main1}.  The rotation $R_\alpha$ admits the Lebesgue measure as the  unique invariant measure and the support of $\sigma_f$ is
	contained in $\alpha \mathbb{Z}$. Actually it is easy to check that
	$$
	\sigma_f = \sum_{n\in \mathbb{Z}} |\widehat{f}(n)|^2 \delta_{e^{2\pi i\alpha n}}.
	$$ 
\end{proof}

\medskip

\begin{remark}\label{R1}
The idea of using  the spectral lemma in the proof of Theorem~\ref{main1} comes from \cite{FS1} where the ergodic Hilbert transforms
associated to irrational rotations are studied. 
\end{remark}
 
\medskip

We will apply Theorem~\ref{main1} to several different flows. In Section 5, we apply to all $p$-adic polynomial flows and $p$-adic rational flows. In Section 6, we apply to all automorphisms and some affine maps of $2$-torus with zero topological entropy. In Section 7, we apply to all 
Feigenbaum zero topological entropy flows. In Section 8, we apply to all orientation-preserving irrational circle homeomorphisms.

\section{\bf $p$-Adic Polynomial Flows and $p$-Adic Rational Flows}

We give here some examples of equicontinuous flows in the fields of
$p$-adic numbers. These flows share a very nice minimal decomposition (\cite{FL11}, \cite{FFLW}). We can compare this decomposition with Auslander's star closed decomposition \cite{Auslander}.

Let $p\ge 2$ be a prime number and let $\mathbb{Z}_p$ be the ring of $p$-adic integers.
Let $\mathbb{Z}_p[x]$ be the ring of all polynomials with coefficients in $\mathbb{Z}_p$.
Then every polynomial $P\in \mathbb{Z}_p[x]$ defines a flow $\mathcal{X}=(\mathbb{Z}_p, P)$. It is trivial that
 $P$ is $1$-Lipschitz function, i.e.
$$
|P(x) - P(y)|_p\le |x -  y|_{p}  \quad (\forall x, y\in \mathbb{Z}_p)
$$
where $|\cdot|_p$ is the $p$-adic norm on $\mathbb{Z}_p$. So, the flow $\mathcal{X}=(\mathbb{Z}_p, P)$
is equicontinuous. In particular, the adding machine $add (x)=x +1 \in \mathbb{Z}_p[x]$ is a special case
and it is actually an equicontinuous homeomorphism.

Corollary~\ref{M2} gives us the following result.

\medskip
\begin{corollary}~\label{pp}
Any  oscillating sequence ${\bf c} =(c_{n})$ satisfying the growth condition (\ref{growth})
is linearly disjoint from all $p$-adic polynomial flows $\mathcal{X}=(\mathbb{Z}_p, P)$ for $P\in \mathbb{Z}_p[x]$.\end{corollary}

It was proved in \cite{FL11} that for a polynomial flow $\mathcal{X}=(\mathbb{Z}_p, P)$ as in Corollary~\ref{pp}, we have the minimal
decomposition
$$
    \mathbb{Z}_p = \mathcal{P} \sqcup\mathcal{M}\sqcup \mathcal{B}
$$
where $ \mathcal{P} $ is the finite set consisting of all periodic points of
$P$, $ \mathcal{M} = \bigsqcup_i \mathcal{M} _i$ is the union of all (at most countably
many) clopen invariant sets such that each $\mathcal{M}_i$ is a finite union
of balls and each subsystem $P: \mathcal{M}_i \to \mathcal{M}_i$ is minimal, and
points in $\mathcal{B}$ lie in the attracting basin of a periodic orbit or of
a minimal sub-flow. As we proved in Corollary~\ref{M2}, points in $\mathcal{B}$
are mean attracted to a periodic cycle or a minimal set $M_i$.
\medskip

More generally, we can consider $p$-adic rational flow in $p$-adic field. Let $\mathbb{Q}_p$ be the field of $p$-adic numbers.
 Any point in the projective line $\P$ may be given in homogeneous coordinates by a pair
$[x_1 : x_2]$
of points in $\Qp$ which are not both zero. Two such pairs $[x_1 : x_2]$ and $[\lambda x_1 : \lambda x_2]$
with nonzero factor $\lambda \in \Q_p^*$ are identified.
The field $\Qp$ may be identified with the subset of $\P$ given by
$$\left\{[x : 1] \in \mathbb P^1(\Qp) \mid x \in \Qp\right\}.$$
This subset covers all points  in $\P$ except one: the point  of infinity, which may be given as
$\infty = [1 : 0].$
The spherical metric defined on $\P$ is analogous to the standard spherical metric on the
Riemann sphere. If $u=[x_1,y_1]$ and $v=[x_2,y_2]$ are two points in $\P$, we define
 $$\rho(u,v)=\frac{|x_1y_2-x_2y_1|_p}{\max\{|x_1|_{p},|y_1|_{p}\}\max\{|x_2|_{p},|y_{2}|_{p}\}}$$
 or, viewing $\mathbb{P}^{1}(\Qp)$ as $\Qp\cup\{\infty\}$, for $z_1,z_2 \in \Qp\cup \{\infty\}$
 we define
 $$\rho(z_1,z_2)=\frac{|z_1-z_2|_{p}}{\max\{|z_1|_{p},1\}\max\{|z_2|_{p},1\}}  \qquad\mbox{if~}z_{1},z_{2}\in \Qp,$$
 and
 $$\rho(z,\infty)=\left\{
                    \begin{array}{ll}
                      1, & \mbox{if $|z|_{p}\leq 1$;} \\
                      1/|z|_{p}, & \mbox{if $|z|_{p}> 1$.}
                    \end{array}
                  \right.
 $$
Remark that the restriction of the spherical metric on the ring $\Zp:=\{x\in \Q_p, |x|\leq 1\}$
of $p$-adic integers is same to the metric induced by the absolute value $|\cdot|_p$.

A rational map $R\in \Qp(z)$ induces a map  on  $\P$, which we still denote as $R$.
Rational maps  are  always Lipschitz continuous on $\P$
with respect to the spherical metric (\cite[Theorem 2.14]{SilvermanGTM241}).
Rational maps  with  good reduction  are  $1$-Lipschitz  continuous
(\cite[p. 59]{SilvermanGTM241}) in the sense that
$$
\rho(R(u),R(v))\leq \rho(u,v), \quad \forall u,v \in \P.
$$

\medskip
\begin{corollary}~\label{rp}
Any oscillating sequence ${\bf c}=(c_{n})$ satisfying the growth condition (\ref{growth}) is linearly disjoint from all $p$-adic rational flows $\mathcal{X}=(\P, R)$, where $R \in \mathbb{Q}_p(x)$ is a rational map  with good reduction.
\end{corollary}

It was proved in~\cite{FFLW} that for a rational function with good reduction of order at least $2$,  we have a minimal
decomposition,
$$
    \P = \mathcal{P} \sqcup\mathcal{M}\sqcup \mathcal{B}.
$$
which is similar to the case of $p$-adic polynomial flow.

One special case of Corollary~\ref{rp} is that 

\medskip
\begin{corollary}~\label{psarnak}
The M\"obius function is linearly disjoint from all $p$-adic polynomial flows $(\mathbb{Z}_p, P)$ and all $p$-adic rational flow  with good reduction $(\mathbb{P}^1(\mathbb{Q}_p), R)$.
\end{corollary}

\section{\bf Automorphisms on $2$-Torus with Zero Topological Entropy}

Let
$$
\mathbb{T}^2=\mathbb{R}^2/\mathbb{Z}^2
$$
be the $2$-torus. Let $SL(\mathbb{Z},2)$ be the space of all matrices $A$ of integral entries such that $\det A=\pm 1$.
Then all $A{\bf x}$ of $\mathbb{T}^2$ for $A\in SL(\mathbb{Z},2)$ represent all automorphisms of $\mathbb{T}^{2}$.
Affine maps on $\mathbb{T}^{2}$ are of the form
$$
T({\bf x}) = A{\bf x} + {\bf b}
$$
where $A$ is an automorphism of $\mathbb{T}^2$ and ${\bf b}$ is an element in $\mathbb{T}^2$.

Considered $T$ as a flow on $\mathbb{T}^2$,  the entropy is equal to the logarithm
of the sum of the modulus of all eigenvalues of $A$ of modulus strictly larger than $1$.
Thus $T$ is of zero topological entropy if and only if all eigenvalues of $A$ are of modulus $1$.

Let $\|\cdot\|_{\mathbb{R}^2}$ be the Euclidean norm on $\mathbb{R}^2$. Recall that
the induced metric on  $\mathbb{T}^2$ is defined by
$$
    \|{\bf x}-{\bf y}\|_{\mathbb{T}^2} = \inf_{{\bf n}\in \mathbb{Z}^2} \|{\bf x}-{\bf y}-{\bf n}\|_{\mathbb{R}^2}
$$
Also recall that the Hermitian norm on $\|\cdot\|_{\mathbb{C}^2}$ on $\mathbb{C}^2$ is defined by
$$
    \|{\bf z}\|_{\mathbb{C}^{2}} = z_1 \overline{z}_1 + z_2 \overline{z}_2, \quad {\bf z} = (z_1, z_2)\in \mathbb{C}^2.
$$
The space $\mathbb{R}^2$ is considered as a subspace of $\mathbb{R}^2$. It is clear that
for ${\bf x}\in \mathbb{R}^2$, we have
$$
  \|{\bf x}\|_{\mathbb{T}^2}\le \|{\bf x}\|_{\mathbb{R}^2}= \|{\bf x}\|_{\mathbb{C}^2}.
$$
The following proposition actually hods in higher dimension, but we just state in the dimension $2$.

\begin{proposition}\label{diag} Let $A\in SL(\mathbb{Z}, 2)$. Suppose that $A$ is diagonalizable in the field of complex numbers and have all its eigenvalues of modulus $1$. Then there exists a constant $C$ such that for all ${\bf n}\in \mathbb{Z}$
and all ${\bf x}\in \mathbb{T}^2$ we have
$$
  \|A^n {\bf x}\|_{\mathbb{T}^2}\le C \|{\bf x}\|_{\mathbb{T}^2}.
$$
Thus the flow $\mathcal{X}= (\mathbb{T}^{2}, x \mapsto A{\bf x} +{\bf b})$ for every ${\bf b}\in \mathbb{T}^{2}$ is equicontinuous.
\end{proposition}

\begin{proof} Let $A=PDP^{-1}$ where $D$ is the diagonal matrix with the  eigenvalues $\lambda_1, \lambda_{2}$ of $A$
on the diagonal and $P$ is an invertible complex matrix. Then
$$
  \|A^n {\bf x}\|_{\mathbb{R}^2}=  \|P D^n P^{-1}{\bf x}\|_{\mathbb{R}^2}
  = \|P D^n P^{-1}{\bf x}\|_{\mathbb{C}^2}
  \le \|P\|_{\mathbb{C}^2} \|D^n {\bf z}\|_{\mathbb{C}^2}
  $$
  where ${\bf z}= P^{-1}{\bf x}=(z_{1},z_{2})$. However
$$
   \|D^n {\bf z}\|_{\mathbb{C}^2}^2= \sum_{j=1}^{2} |\lambda_j^n z_j|^2 = \sum_{j=1}^{2} |z_j|^2
   = \|P^{-1} {\bf x}\|_{\mathbb{C}^2}^2.
$$
Thus
$$
    \|A^n {\bf x}\|_{\mathbb{R}^2} \le \|P\|_{\mathbb{C}^2} \|P^{-1}\|_{\mathbb{C}^2} \|{\bf x}\|_{\mathbb{R}^2}.
$$
Let ${\bf n}\in \mathbb{Z}^{2}$ such that $\|{\bf x}\|_{\mathbb{T}^2} = \|{\bf x}-{\bf n}\|_{\mathbb{R}^2}$.
Replace ${\bf x}$ by ${\bf x} -{\bf n}$ in the above
inequality, we get
$$
    \|A^n ({\bf x} -{\bf n})\|_{\mathbb{R}^2} \le \|P\|_{\mathbb{C}^2} \|P^{-1}\|_{\mathbb{C}^2} \|{\bf x}\|_{\mathbb{T}^2}.
$$
Since $A^n {\bf n} \in \mathbb{Z}^2$ for all $n\geq 1$, finally we get for $C= \|P\|_{\mathbb{C}^2} \|P^{-1}\|_{\mathbb{C}^2}$,
$$
 \|A^n {\bf x}\|_{\mathbb{T}^2} \le C \|{\bf x}\|_{\mathbb{T}^2}.
$$
\end{proof}

Combining this proposition and Corollary~\ref{M2}, we have that

\medskip
\begin{corollary}~\label{diagcoro}
Suppose $A\in SL(\mathbb{Z},2)$ is diagonalizable in complex field. 
And all its eigenvalues have modulus $1$. Then any oscillating sequence ${\bf c}=(c_{n})$ satisfying the growth condition (\ref{growth})
 is linearly disjoint from the flow
$$
\mathcal{X} = (\mathbb{T}^{2}, T{\bf x}: =A{\bf x}+{\bf b})
$$
for every ${\bf b}\in \mathbb{T}^{2}$.
\end{corollary}

Now suppose $A$ is not diagonalizable. In this case,  $A$ has a double eigenvalue $1$ or $-1$ and $\det A=1$.
A $2\times 2$-square matrix is called a {\em module matrix} if all entries are integers and its determinant is $1$.

\medskip
\begin{proposition} \label{nondiag} Let $M$ be a modular matrix with $\pm 1$ as a double eigenvalue. There exists a modular
matrix $P$ such that
$$
P^{-1} M P=\pm T_{t}
$$
where $T_{t}$ is a modular matrix of the form
$$
     T_{t} = \left(
           \begin{array}{cc}
             1 & t \\
             0 & 1 \\
           \end{array}
         \right), \quad t \in \mathbb{Z}.
$$
\end{proposition}

\begin{proof} If the double eigenvalue of $M$ is $-1$. Then $-M$ has a double eigenvalue $1$.
So we only need to prove the case that the double eigenvalue of $M$ is $1$. In this case the trace of $M$ is $2$. So we can write
$$
  M = \left(
           \begin{array}{cc}
             a & b \\
             c & 2-a \\
           \end{array}
         \right),
         \quad
  M -I = \left(
           \begin{array}{cc}
             a-1 & b \\
             c & 1-a \\
           \end{array}
         \right).
$$
The fact $\det M =1$ means
\begin{equation}\label{det}
-bc = (a-1)^2.
\end{equation}

If $b=0$, we must have $a=1$ by (\ref{det}). So
$$
    M = \left(
           \begin{array}{cc}
             1 & 0 \\
             c & 1 \\
           \end{array}
         \right) =  P \left(
           \begin{array}{cc}
             1 & - c \\
             0 & 1 \\
           \end{array}
         \right) P^{-1}, \quad \mbox{\rm with} \  \ P=   \left(
           \begin{array}{cc}
             0 & 1 \\
             -1 & 0 \\
           \end{array}
         \right)
$$
Thus we are done. If $c=0$, we must have $a=1$ and there is nothing to do.

Assume  $b\not=0$ and $c\not=0$. By (\ref{det}), we have $a\not=1$ and  $c = -\frac{(a-1)^2}{b}$.
We are looking for 
$$
{\bf x} =
\left(
  \begin{array}{c}
    x_1 \\
    x_2 \\
  \end{array}
  \right)\in \mathbb{Z}^2
  $$
  which is an  eigenvector associated to $1$, i.e.
$(M-I){\bf x}=0$. This equation is equivalent to
$$
     (a-1) x_1 + b x_2 = 0\quad \hbox{and} \quad c x_{1}+ (1-a) x_{2}=0.
$$
Let $d$ be the gcd of $(a-1)$ and $b$. The above two equations become one equation
\begin{equation}\label{EE1}
    \frac{a-1}{d} x_1 + \frac{b}{d} x_2=0  
\end{equation}
We choose an integral solution to the first equation of (\ref{EE1}):
$$
{\bf x} =
\left(
  \begin{array}{c}
    x_1 \\
    x_2 \\
  \end{array}
\right) =
\left(
\begin{array}{c}
    \frac{b}{d} \\
    -\frac{a-1}{d} \\
  \end{array}
\right).
$$
Then we are looking for a generalized eigenvector 
$$
{\bf y} =
\left(
  \begin{array}{c}
    y_1 \\
    y_2 \\
  \end{array}\right)\in \mathbb{Z}^2
  $$
such that $(M-I) {\bf y} = t {\bf x}$ with $t$ an integer to be determined. The first line of the last equation is equivalent to
\begin{equation}\label{EE2}
   (a-1) y_1 + b y_2 = t \frac{b}{d}.
\end{equation}

We claim that $b|d^2$. In fact, $b|(a-1)^2$. Let us factorize $a-1$ into primes:
$$
a-1 = \pm p_1^{\alpha_1} p_2^{\alpha_2} \cdots p_r^{\alpha_r}.
$$
Then we can write
$$
    b = \pm p_1^{\beta_1} p_2^{\beta_2} \cdots p_r^{\beta_r},
    \quad (0\le \beta_j \le 2\alpha_j, 1\le j \le r).
$$
So
$$
     d = p_1^{\alpha_1 \wedge \beta } p_2^{\alpha_2\wedge \beta_2} \cdots p_r^{\alpha_r\wedge \beta_r},
$$
where $\alpha\wedge \beta=\min \{ \alpha, \beta\}$.
Then $b|d^2$ is equivalent to say that $\beta_j \le 2 (\alpha_j \wedge \beta_j)$ for all
$1\le j \le r$, which is true because  $0\le \beta_j \le 2\alpha_j$.

Take $t = b^{-1} d^2 \in \mathbb{Z}$. Then the equation (\ref{EE2}) becomes
\begin{equation}\label{EE3}
    \frac{a-1}{d} y_1 + \frac{b}{d} y_2 = 1.
\end{equation}
By the B\'ezout theorem, the equation (\ref{EE3}) admits a solution $(y_1, y_2)$ in $\mathbb{Z}^2$ because $\big( (a-1)/d, b/d\big)=1$.
Let $P=({\bf x}, {\bf y})$ be the $2\times 2$ matrix with columns ${\bf x}$ and ${\bf y}$. Let
${\bf u}=(x_1, y_1)$ and ${\bf v} =(x_2, y_2)$ be the two rows of $P$. The first equation of (\ref{EE1}) and  (\ref{EE3})
can be rewritten as
\begin{equation}\label{EE4}
     {\bf e}_{2}:= \frac{a-1}{d} {\bf u} + \frac{b}{d}{\bf v} = (0,1)
\end{equation}
Then
  $$
    \det P = \det ({\bf u}, {\bf v}) = \frac{d}{b} \det ({\bf u}, d^{-1} b {\bf v})
    = \frac{d}{b}\det ({\bf u}, {\bf e}_2)= \frac{d}{b} \cdot \frac{b}{d}=1
  $$
  where the relation (\ref{EE4}) is used for the third equality. This implies that $P$ (as well as $P^{-1}$) is a positive modular matrix and
  $$
  P^{-1} MP = T_{t}
  $$
  for some $t=b^{-1}d^{2} \in \mathbb{Z}$.
\end{proof}

Following Proposition~\ref{diag}, Proposition~\ref{nondiag}, Example~\ref{2ex} and Theorem~\ref{main1}, we have that

\medskip
\begin{theorem}~\label{nondiagthm}
Suppose $A\in SL(\mathbb{Z},2)$ and suppose all its eigenvalues have modulus $1$.
Then the flow $\mathcal{X} = (\mathbb{T}^{2}, T{\bf x}=A{\bf x})$ is {\bf MMA} and {\bf MMLS}.
 \end{theorem}

A consequence of Theorem~\ref{main1} and Theorem~\ref{nondiagthm} is the corollary.

\medskip
\begin{corollary}~\label{nondiagcoro}
Suppose $A\in SL(\mathbb{Z},2)$ and suppose all its eigenvalues have modulus $1$.
Any oscillating sequence ${\bf c}=(c_{n})$ satisfying the growth condition (\ref{growth})
is linearly disjoint from the flow $\mathcal{X}=(\mathbb{T}^{2}, T{\bf x}=A{\bf x})$.
\end{corollary}

The disjointness of the M\"{o}bius function from the flow in Corollary~\ref{nondiagcoro} was proved by Liu and Sarnak  in~\cite{LS}.

Let us make a remark about the number $t\in \mathbb{Z}$ in Proposition~\ref{nondiag}. It may not be unique but can not also arbitrary.
Suppose we have another $t'\in \mathbb{Z}$ such that Proposition~\ref{nondiag} holds. Then we have a modular matrix $P$ such that
$$
T_{t}P=PT_{t'}.
$$
Let $(x,y)$ be the first row of $P$ and $(u, v)$ be the second row of $P$. Then we have that
$$
tu=0, \;\; tv=t'x,\;\; t'u=0.
$$
This implies that $t=0$ if and only if $t'=0$. If $t\not=0$ and $t'\not=0$, then $v=(t'/t)x$ and $u=0$. Since $det P=1$, this implies that $x^{2} =t/t'$.
That is $t/t'$ must be the square of an integer. This implies that $T_{t}$ and $T_{t'}$ are (linearly) conjugate dynamical systems on $\mathbb{T}^{2}$ if and only if $t/t'$ is the square of an integer. For example,
$$
 \left(
             \begin{array}{cc}
               1 & 2 \\
               0 & 1 \\
             \end{array}
           \right)
      \quad      \hbox{and}\quad
           \left(
              \begin{array}{cc}
                1 & 1 \\
                0 & 1 \\
              \end{array}
            \right).
$$
are not conjugate flows on $\mathbb{T}^{2}$.

\medskip
Furthermore, let us give an example to which Proposition~\ref{nondiag} applies:
Consider the modular matrix
$$
   M= \left(
             \begin{array}{cc}
               -5 & 6 \\
               -6 & 7 \\
             \end{array}
           \right),
           \quad
           P= \left(
             \begin{array}{cc}
               1 & 0 \\
               1 & 1 \\
             \end{array}
           \right),
           \quad T_{6}= \left(
             \begin{array}{cc}
               1 & 6 \\
               0 & 1 \\
             \end{array}
           \right),
$$
Then we have
$$
P^{-1}MP =T_{6}.
$$

\medskip
When $A$ is not diagonalizable, we show in the following example that there is an oscillating sequence and an affine map 
${\bf x}:\to A{\bf x} +{\bf b}$ on $2$-tours such that this oscillating sequence is not linearly disjoint from the flow defined by this affine map.
 
\medskip
\begin{example}
 Let
$$
 A = \left(
             \begin{array}{cc}
               1 & 0 \\
               1 & 1 \\
             \end{array}
           \right)
, \qquad
{\bf b}=\left(
             \begin{array}{c}
               \alpha \\
               0 \\
             \end{array}
           \right)
$$
where $0<\alpha<1$ is an irrational number. 
Define 
$$
T_{A,b}({\bf x}) = A{\bf x} +{\bf b}: {\mathbb T}^{2}\to {\mathbb T}^{2}.
$$
The sequence ${\bf c}=(e^{- \pi i n^2 \alpha})$ is an oscillating sequence (refer to Example 3) satisfying the growth condition (\ref{growth}). 
But it is not linearly disjoint 
from the flow ${\mathcal X}=({\mathbb T}^{2}, T_{A,b})$ (whose topological entropy is zero).
\end{example}

\begin{proof}
Let 
$ {\bf x}= (x, y)^{t} \in  {\mathbb T}^{2}$.
Observe that
$$
T^n {\bf x} = \left(
             \begin{array}{c}
              x + n \alpha \\
              \frac{1}{2}n(n-1) \alpha + n x + y \\
             \end{array}
           \right).
$$
If we take the continuous function $f(x, y)=e^{2\pi i y}$, then 
we have 
$$
   \frac{1}{N}\sum_{n=1}^N e^{- \pi i n^2 \alpha} f(T^n(\alpha/2, 0))
   = \frac{1}{N}\sum_{n=1}^N e^{- \pi i n^2 \alpha} e^{ \pi i n^2 \alpha} =1
$$
which does not  tend to zero.
\end{proof}

\medskip
\begin{remark}
As a skew product of an irrational rotation, the  map $T_{A,b}$ is uniquely ergodic with the
Lebesque measure as the invariant measure having full support. This implies that $T_{A,b}$ is minimal, so {\bf MMA}.
But it is not {\bf MMLS} (equivalently {\bf MLS}). Otherwise, Theorem \ref{main1} would imply that the sequence ${\bf c}$ defined by $c_{n}=e^{- \pi i n^2 \alpha}$ is linearly disjoint from the flow
${\mathcal X} =({\mathbb T}^{2}, T_{A,b})$.
\end{remark}

\medskip
\begin{remark}~\label{higher}
In~\cite{LS}, Liu and Sarnak has proved that the M\"{o}bius function $\mu (n)$ is linearly disjoint from all affine flows 
$\mathcal{X} =( \mathbb{T}^{2}, T{\bf x} =A{\bf x} +{\bf b})$ whose topological entropy is zero, where 
$A\in SL(\mathbb{Z},2)$ and ${\bf b}\in \mathbb{T}^{2}$ (the result holds in high dimensions). This result and the above example indicate that in order for
an oscillating sequence ${\bf c}=(c_{n})$ to be linearly disjoint from all zero entropy flows, it may not be oscillating only in the first order but also oscillating in any higher order, that is, 
$$
    \lim_{N \to \infty} \frac{1}{N} \sum_{n=1}^N c_n e^{2\pi i P(n)}=0
$$
for all real polynomials $P$. The M\"obius function and the von Mangoldt function, which is defined as $\Lambda (n)= \log p$ if $n=p^{k}$ for some prime number and integer $k\geq 1$ and $0$ otherwise, share the above oscillating property for any high order as shown in~\cite{Hua}. 
\end{remark}

\section{\bf Feigenbaum Zero Topological Entropy Flows}

A famous zero topological entropy flow has been studied extensively for last fourty years is the so-called Feigenbaum quadratic-like map.
It is a non-linear dynamical system and does not have the mean-equicontinuous property on the whole space.
In this section, we will prove that a Feigenbaum
zero topological entropy flow is minimally mean attractable ({\bf MMA}) and minimally mean-L-stable ({\bf MMLS}).  
Thus  Sarnak's conjecture is true for such a dynamical system.
We use $ent (T)$ to denote the topological entropy of a dynamical system $T: X\to X$. For the definition of the topological entropy, the reader can refer to~\cite{Bo,Si,Pa}.

Let $I=[-1,1]$.
Suppose
$$
T_{t} (x) =t-(1+t) x^{2}: I\to I.
$$
It is a family of quadratic polynomials with parameter $-1/2 \leq t \leq 1$.
Then $-1$ is a fixed point of $T_{t}$ and $T_{t}$ maps $1$ to $-1$. The map $T_{t}$ has a unique critical point $0$
and it is non-degenerate, that is, $T_{t}'(0)=0$ and $T_{t}''(0)\not=0$.

When $t_{0}=-1/2$, $T_{t_{0}}$ has $-1$ as a parabolic fixed point, that is, $T_{t_{0}}'(-1)=1$ (see Figure 1). And $T_{t_{0}}$
has no other fixed point. One can see easily that
$$
T_{t_{0}}^{n}  (x) \to -1 \quad \hbox{ as $n\to \infty$}, \quad \forall \; x\in I .
$$
So $ent (T_{t_{0}}) =0$ and in this case,
$K_{0}=\{-1\}$ is the only minimal set and $\hbox{\rm Basin}(K_{0})=[-1,1]$. Thus the flow $\mathcal{X}=([-1,1], T_{t_{0}})$ is {\bf MMA} and  {\bf MMLS}.

\begin{figure}[htp]~\label{fig1}
\centering{
\includegraphics[scale=0.6]{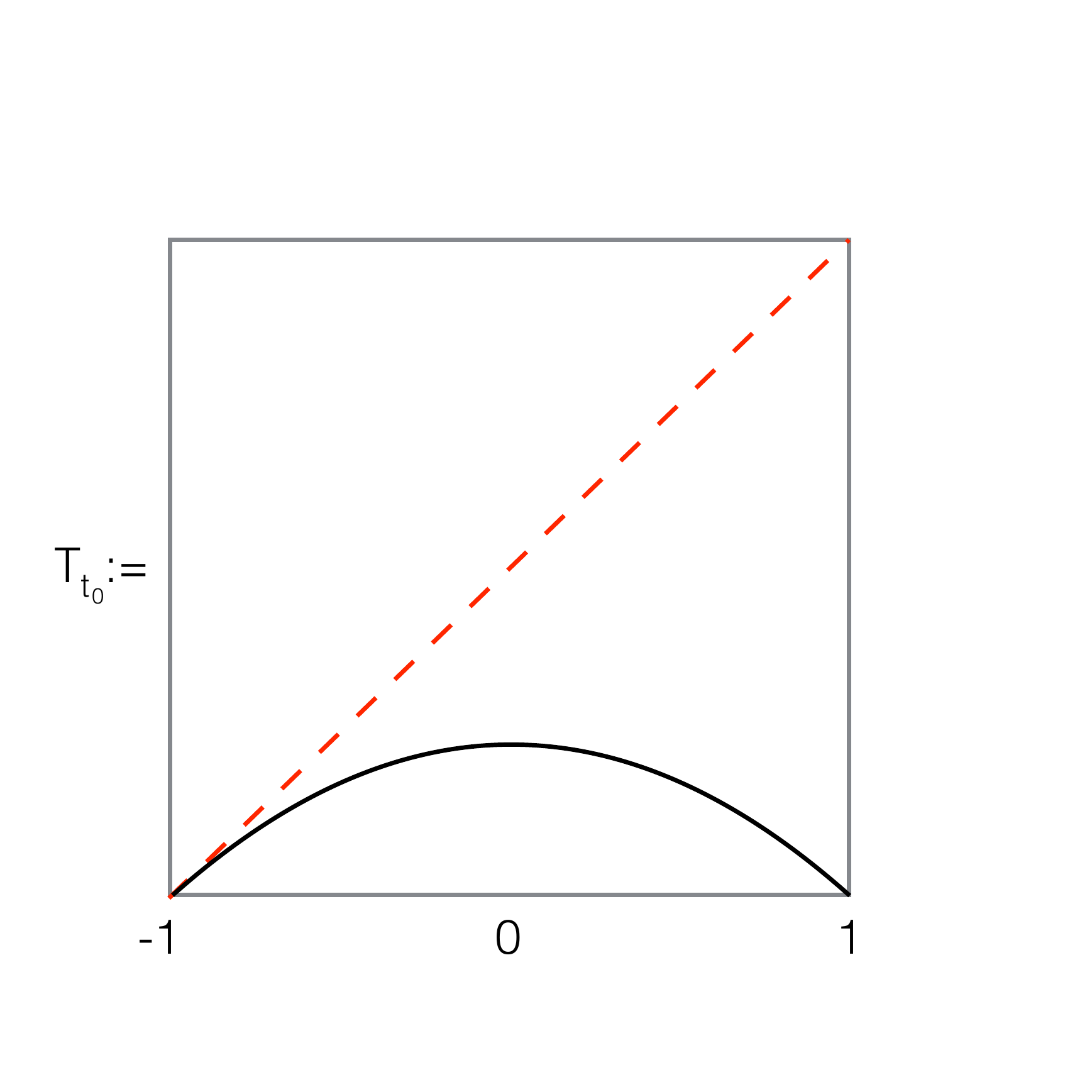}}
\vspace*{-40pt}
\caption{}
\end{figure}

When  $t>t_{0}$ but close to $t_{0}$, $-1$ becomes a repelling fixed point of $T_{t}$, that is, $T_{t}' (-1)>1$,
but $T_{t}$ has another fixed point $p_{t}$ which is an attracting fixed point, that is, $T_{t} (p_{t}) =p_{t}$ and $|T_{t}' (p_{t})| <1$ (see Figure 2).
Furthermore, we have that
$$
T_{t}^{n}  (x) \to p_{t} \quad \hbox{as $n\to \infty$}, \;\;\forall \; x\in (-1,1) .
$$
\begin{figure}[htp]~\label{fig2}
\centering{
\includegraphics[scale=0.6]{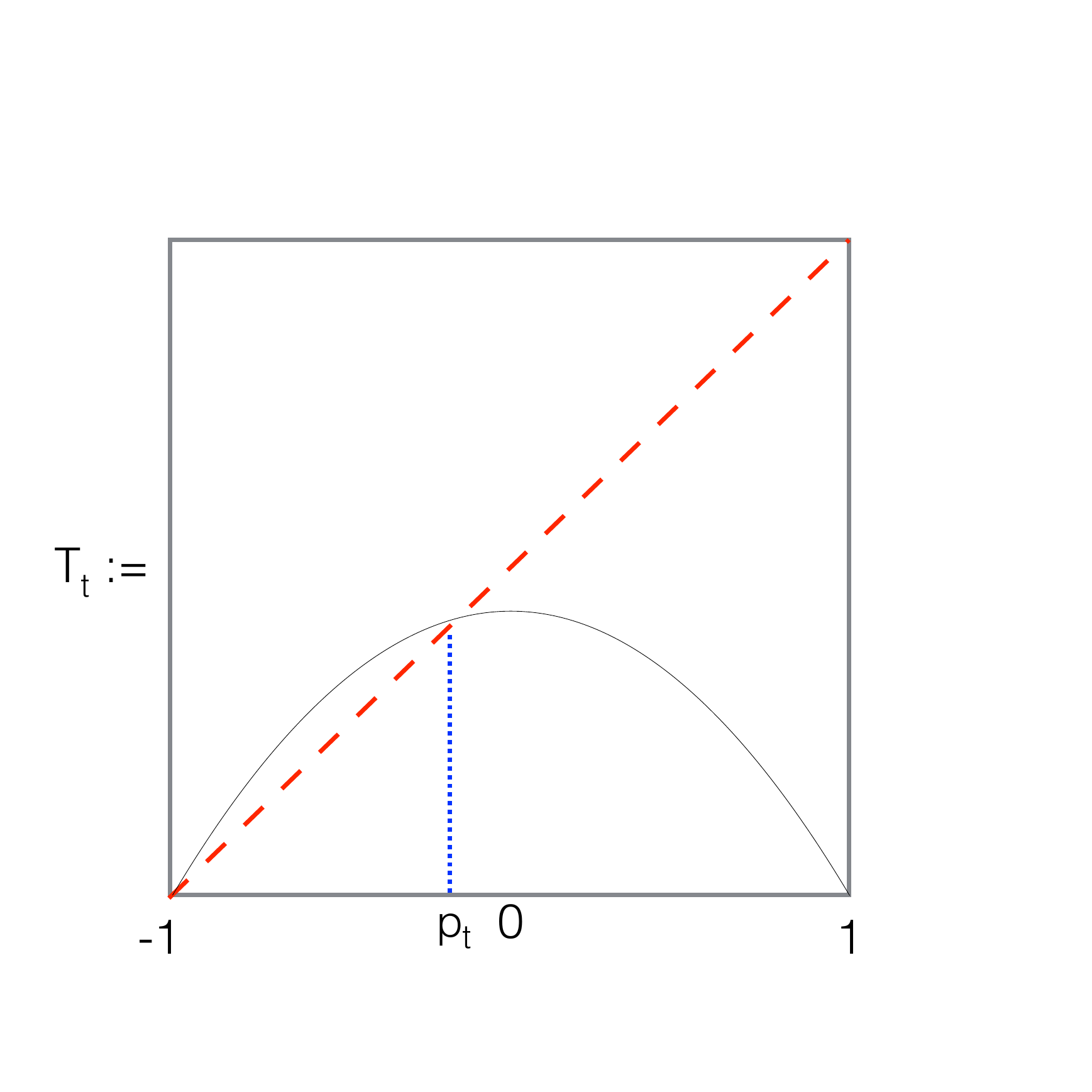}}
\vspace*{-10pt}
\caption{}
\end{figure}

This phenomenon keeps until $p_{t}$ becomes a parabolic fixed point at $t_{1}$, that is, $T_{t_{1}} '(p_{t_{1}}) =-1$. So, for $t_{0}< t\leq t_{1}$,  we have that
$ent (T_{t}) =0$ and
$T_{t}$ has two minimal subsets $K_{0}=\{-1\}$ and $K_{1} =\{ p_{t}\}$ and
$\hbox{\rm Basin}(K_{0})=\{-1,1\}$ and $\hbox{\rm Basin}(K_{1})=(-1,1)$.
This proves that the flow $\mathcal{X}=([-1,1], T_{t})$ is {\bf MMA} and {\bf MMLS}.

When  $t>t_{1}$ but close to $t_{1}$, $p_{t}$ becomes a repelling fixed point of $T_{t}$, that is, $T_{t} (p_{t}) =p_{t}$ and $|T_{t} '(p_{t})|>1$,
but there is a periodic cycle of period $2$ appearing near $p_{t}$, which we denote as $\{ p_{t,2,0}, p_{t,2,1}\}$.
This cycle is attractive, that is,
$$
T_{t} (p_{t,2.0})= p_{t, 2,1}, \quad T_{t} (p_{t,2.1})= p_{t, 2,0}
$$
and
$$
|(T_{t}^{2})' (p_{t,2.0})| = |(T_{t}^{2})' (p_{t,2.1}) |<1
$$
(see Figure 2).
Moreover,
$$
T_{t}^{n}  (x) \to \{ p_{t,2,0}, p_{t,2,1}\}\;\; \hbox{ as $n\to \infty$}, \;\;\; \forall x\in I\setminus \big( \cup_{m=1}^{\infty} T_{t}^{-m} (p_{t}) \cup \{-1,1\}\big).
$$
This phenomenon keeps until the cycle $\{ p_{t,2,0}, p_{t,2,1}\}$ becomes a parabolic periodic cycle of period $2$ at $t_{2}$, that is,
$$
|(T_{t_{1}}^{2})' (p_{t_{2},2,0})| =|(T_{t_{1}}^{2})' (p_{t_{2},2,1})|=1.
$$
So, for $t_{1}< t\leq t_{2}$,  we have that
$ent (T_{t}) =0$ 
and $T_{t}$ has three minimal subsets
$K_{0}=\{-1\}$ and $K_{1}=\{p_{t}\}$ and $K_{2} =\{ p_{t,2,0}, p_{t,2,1}\}$ with
$$
\hbox{\rm Basin}(K_{0})=\{-1,1\},  \;\;\; \hbox{\rm Basin}(K_{1})=\cup_{n=0}^{\infty} T^{-n} (p_{t}),
$$
and
$$
\hbox{\rm Basin}(K_{2})=(-1,1)\setminus \hbox{\rm Basin}(K_1).
$$
We have proved that the flow $\mathcal{X}=([-1,1], T_{t})$ is {\bf MMA} and {\bf MMLS} for $t_{1}<t\leq t_{2}$.


\begin{figure}[htp]~\label{fig3}
\centering{
\includegraphics[scale=0.8]{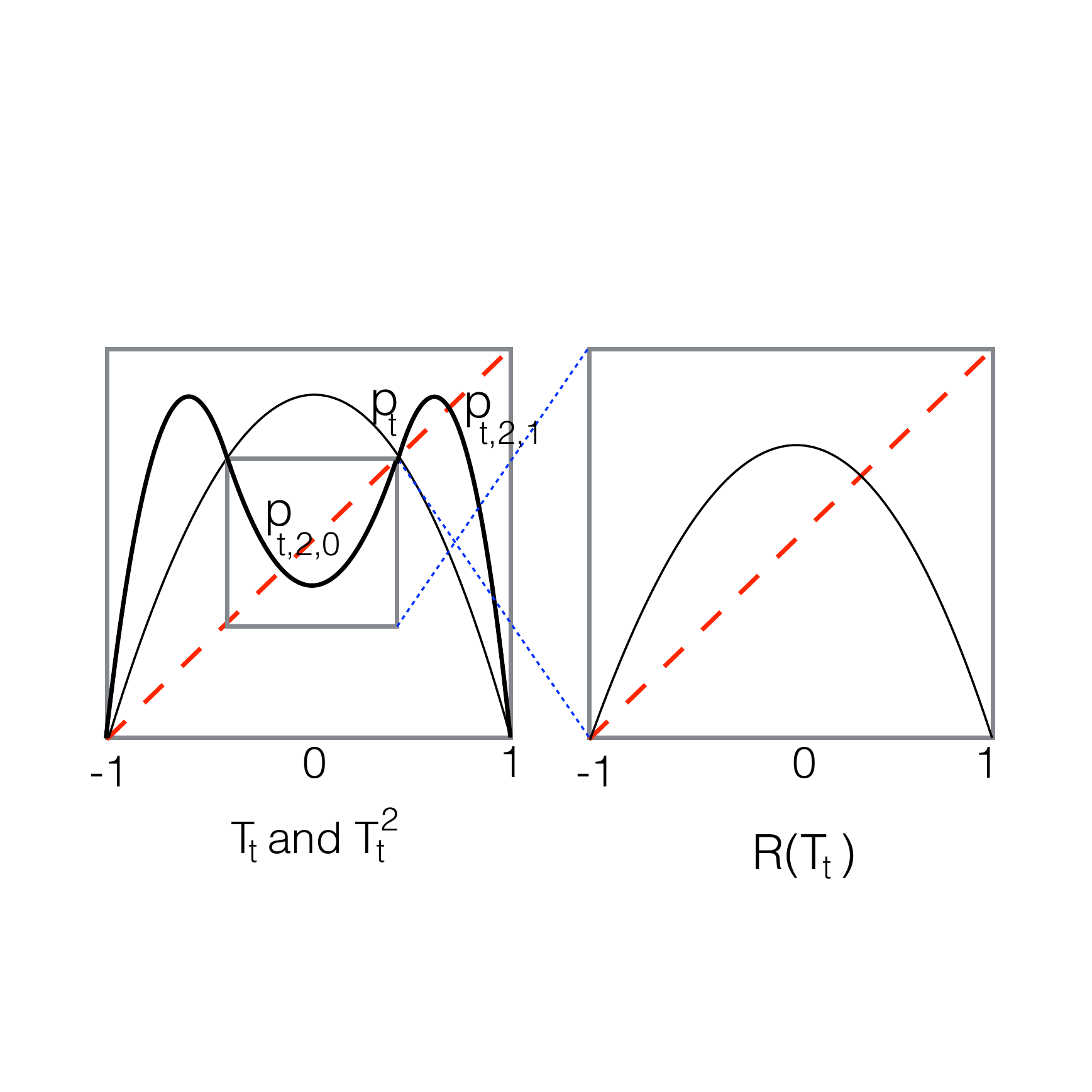}}
\vspace*{-100pt}
\caption{}
\end{figure}

To see how the periodic cycle of period $2$ was born, we can use the period doubling operator ${\mathcal R}$ as follows. When $t_{1}\leq t\leq t_{2}$,
let $\beta = p_{t}>0$. Consider the map
$$
{\mathcal R} (T_{t}) (x) = -\beta^{-1} T_{t}^{2} (-\beta x): I\to I.
$$
It is not a quadratic polynomial but has a graph similar to $T_{t}$ for $t_{0} \leq t\leq t_{1}$.
So when $t=t_{1}$, ${\mathcal R} (T_{t})$ has a parabolic fixed point $-1$ and
when $t_{1}<t<t_{2}$, ${\mathcal R} (T_{t})$ has $-1$ as a repelling fixed point and there is an attractive fixed point $q_{t}$.
When $t=t_{2}$,  ${\mathcal R} (T_{t})$ has $-1$ as a repelling fixed point and $q_{t}$ becomes a parabolic fixed point.
Note that
$$
p_{t,2,0} =-\beta q_{t} \quad \hbox{and}\quad p_{t,2,1}=T_{t} (p_{t,2,0})
$$
become a cycle of periodic points of period $2$ for $T_{t}$.

In general, repeat the above process, we have the following dichotomy.
There is an increasing sequence
$$
t_{0} < t_{1}< \cdots < t_{n} < t_{n+1} <\cdots <1,
$$
for any $t_{n-1}<t \leq t_{n}$, $T_{t}$ has a repelling periodic cycle
$$
Orb (i) =\{ p_{t, 2^{i}, k}\}_{k=0}^{2^{i-1}}
$$
of period $2^{i}$ for any $0\leq i<n$ and an attractive (or parabolic but semi-attractive) periodic cycle
$$
Orb (n) =\{ p_{t, 2^{n}, k}\}_{k=0}^{2^{n-1}}
$$
such that
$$
T_{t}^{n}  (x) \to Orb (n)\;\; \hbox{ as $n\to \infty$}, \;\;\; \forall x\in I\setminus P_{t}
$$
where $P_{t}= \cup_{m=0}^{\infty} T^{-m}_{t} (\cup_{i=0}^{n-1} Orb (i))$ is a countable subset of $I$.
So we have that $ent (T_{t}) =0$ for $t_{0}\leq t<t_{\infty}$,
where
$$
t_{\infty}=\lim_{n\to \infty} t_{n}.
$$
In the case $t_{n-1}<t\leq t_{n}$, $T_{t}$ has $n+1$ minimal subset sets, $K_{0} =\{-1\}$, $K_{1} =\{p_{t}\}$, and
$$
K_{i+1} =\{ p_{t,2^i, k}\}_{k=0}^{2^{i}-1}, \quad i=1, 2, \cdots, n-1.
$$
with $\hbox{\rm Basin}(K_{0})=\{-1,1\}$ and
$$
\hbox{\rm Basin}(K_{i})=\cup_{n=0}^{\infty} T^{-n} (K_{i}), \quad i=1, 2, \cdots, n-1
$$
and
$$
\hbox{\rm Basin}(K_{n})=[-1,1]\setminus \cup_{i=0}^{\infty} \hbox{\rm Basin}(K_{i}).
$$
Thus we have proved that the flow $\mathcal{X}=([-1,1], T_{t})$ is {\bf MMA} and {\bf MMLS}.

Concluding from the above, we have that

\medskip
\begin{theorem}~\label{easyone}
For any $-1/2\leq t<t_{\infty}$, the flow 
$$
\mathcal{X} =(I, T_{t} (x) =t-(1+t)x^{2})
$$ 
is {\bf MMA} and {\bf MMLS}.
\end{theorem}

\medskip
\begin{corollary}~\label{prefeig}
Any oscillating sequence ${\bf c}=(c_{n})$ satisfying the growth condition (\ref{growth})
 is linearly disjoint from 
 $$
 \mathcal{X}= (I, T_{t} (x) =t-(1+t)x^{2})
 $$
 for any $-1/2\leq t<t_{\infty}$
 \end{corollary}
 
 A special case in Corollary~\ref{prefeig} is the M\"obius function. 

The quadratic polynomial $T_{t_{\infty}}$ is called the {\em Feigenbaum quadratic polynomial}.
This polynomial has been studied for the past 40 years since it was discovered by Feigenbaum~\cite{Fe1,Fe2} and, independently, by Coullet and Tresser~\cite{CT}, in 1970s. 
The following theorem can be found in the literature~\cite{CEBook,JBook,MBook,MVBook}.

\medskip
\begin{theorem}~\label{fei}
The quadratic polynomial $T_{t_{\infty}}$ on $I$ as a dynamical system has two repelling fixed points $-1$ and $p_{t_{\infty}}$ and a repelling periodic cycle
$$
Orb (n) =\{ p_{t, n, k}\}_{k=0}^{2^{n-1}}
$$
of period $2^{n}$ for every $n\geq 1$.
There is a Cantor set $\Lambda$ such that
$$
T_{t_{\infty}}^{n} (x) \to \Lambda \;\; \hbox{ as $n\to \infty$},\;\;\; \forall \; x\in I\setminus P_{t_{\infty}},
$$
where
$$
P_{t_{\infty}}=\cup_{m=0}^{\infty} T^{-m} \Big( \{ -1,  p_{t_{\infty}}\} \cup \big(\cup_{n=0}^{\infty} Orb (n)\big) \Big)
$$
is a countable subset of $I$.
Moreover, the topological entropy
$$
ent (T_{t_{\infty}})=0.
$$
\end{theorem}

In the rest of this section, we will give a rigorous proof of that $T_{t_{\infty}}$ on $I\setminus P_{t_{\infty}}$ is {\bf MLS} (that is, mean-equicontinuous).
Actually, we will prove all Feigenbaum zero topological entropy dynamical systems are {\bf MLS} (that is, mean-equicontinuous) on their corresponding set.

Our strategy in the proof is as follows. First we will give a rigorous proof of that the restriction of
a Feigenbaum zero topological entropy flow to its Cantor set attractor is topologically conjugate
to the adding machine on the $2$-adic symbolic flow (refer to Section 5).
This fact was first observed by Feigenbaum in~\cite{Fe1,Fe2} when he proposed his famous universality theory
in chaos (and, independently, by Coullet and Tresser~\cite{CT}).  Then it was emphasized by Sullivan in his many 
lectures attended by the second author in 1986-1989 and in our paper and book~\cite{JMS,JBook}.  
Our rigorous proof is based on two facts. The first one is the existence of the fixed point of the periodic doubling operator 
which was hot topic during the past 40 years in the research of
complex dynamics and chaos. The existence of the fixed point was first observed by Feigenbaum in~\cite{Fe1,Fe2}
and used to explain his university theory in chaos. Then was proved by Lanford in~\cite{La} with computer-assistant
(but still a legitimate mathematical proof). Furthermore, a conceptual understanding and more rigorous mathematical
proof by using Teichm\"uller theory and following idea of Moscow rigidity theorem was given by very recent work 
of many mathematicians (refer to~\cite{Su,MBook2,JBook,Ly}. The other is the connection between Cantor set attractor of the Feigenbaum fixed point  
and the non-escaping set of an induced expanding flow as studied in our previous paper~\cite{JMS}.

To state our main result, we first define the space of all Feigenbaum quadratic-like maps.
Suppose $T: I\to I$ is a $C^{3}$-map. We call it a {\em quadratic-like map} if
\begin{enumerate}
\item $T(-1) =T(1)=-1$;
\item $T'(0)=0$ and $T''(0)\not=0$;
\item $T'(x) >0$ for $x\in [-1, 0)$ and $T' (x) <0$ for $x\in (0, 1]$;
\item the Schwarzian derivative
$$
S(T) (x):=\frac{T'''(x)}{T'(x)} -\frac{3}{2} \Big(\frac{T''(x)}{T'(x)}\Big)^{2}<0, \quad \forall x\in I.
$$
\end{enumerate}
Let ${\mathcal Q}$ be the space of all quadratic-like maps. Then the topological entropy
$$
ent (T): {\mathcal Q}\to {\mathbb R}^{+}=\{ x\geq 0\}
$$
defines a functional.
Using the theory of kneading sequences of Milnor-Thurston~\cite{MT,CEBook}, the space ${\mathcal Q}$ is divided into two parts,
$$
{\mathcal Q}_{0} =\{T\in {\mathcal Q} \;|\; ent (T) =0\}
$$
and
$$
{\mathcal Q}_{+} =\{ T\in {\mathcal Q} \;|\; ent (T) >0\}
$$
The boundary of both spaces
$$
{\mathcal F}_{\infty} =\partial {\mathcal Q}_{0}=\partial {\mathcal Q}_{+}
$$
is called the {\em space of Feigenbaum quadratic-like maps}. 

\medskip
\begin{remark}
The space of Feigenbaum quadratic-like maps can be defined much more general. For examples, we can assume that the type of critical point is $|x|^{\alpha}$ for some $\alpha >1$ and that the smoothness is $C^{1+Z}$ where $Z$ means the Zygmund modulus of continuity.  See~\cite{JBook} for a more general definition of Feigenbaum folding maps.
To avoid to lost the reader from our main purpose, we use an easy definition of a Feigenbaum map.  What we will prove in Theorem~\ref{main2} works for any Feigenbaum folding map but the argument will be more complicate. 
\end{remark}

We would like to note that the space
${\mathcal F}_{\infty}$ is all quadratic-like maps with the same Feigenbaum kneading sequence which is the kneading sequence of $T_{t_{\infty}}$.
Moreover, we have the following result from the literature~\cite{JBook,MBook,MVBook}.

\medskip
\begin{theorem}~\label{fei1}
Any $T\in {\mathcal F}_{\infty}$ has two fixed points $-1$ and $\beta$ and a repelling periodic cycle
$$
Orb (n) =\{ p_{T, 2^n, k}\}_{k=0}^{2^{n-1}}
$$
of period $2^{n}$ for every $n\geq 1$.
There is a Cantor set $\Lambda_{T}$ such that
$$
T^{n} (x) \to \Lambda_{T}\;\;  \hbox{ as $n\to \infty$}, \;\;\; \forall x\in I\setminus P_{T}
$$
where
$$
P_{T}=\cup_{m=0}^{\infty} T^{-m} \Big(\{-1,\beta\}\cup \big(\cup_{n=1}^{\infty} Orb (n)\big)\Big)
$$
is a countable subset of $I$.
Furthermore, there is a homeomorphism $h$ of $I$ such that
$$
h\circ T=T_{t_{\infty}} \circ h \quad \hbox{on $I$}.
$$
Thus the topological entropy
$$
ent(T)=0.
$$
\end{theorem}

We call ${\mathcal X} =(I, T)$ for $T\in {\mathcal F}_{\infty}$ a Feigenbaum zero topological entropy flow.

\medskip
\begin{theorem}~\label{main2}
Any Feigenbaum zero topological entropy flow $\mathcal{X}$ is {\bf MMA} and {\bf MMLS}.
\end{theorem}

\begin{proof}
We first define the period doubling operator on ${\mathcal F}_{\infty}$ as follows.
Any $T\in {\mathcal F}_{\infty}$  has two fixed points $\alpha =-1$ and $\beta \in (0,1)$. Define
$$
{\mathcal R} (T) = -\beta^{-1} T^{2} (-\beta x), \quad x\in I.
$$
Then ${\mathcal R} (T) \in {\mathcal F}_{\infty}$. Thus
$$
{\mathcal R} : {\mathcal F}_{\infty}\to  {\mathcal F}_{\infty}
$$
defines an operator which is called the period doubling operator.

As we have remarked in the paragraph after Theorem~\ref{fei},
the period doubling operator has a unique fixed point $G\in {\mathcal F}_{\infty}$, that is,
$$
{\mathcal R} (G) =G.
$$
With the existence of this fixed point, we continue our proof as follows.
Consider the critical value $c_{1}= G(0)$ and $c_{2} =G^{2} (0)$, $c_{3} =G^{3}(0)$, $c_{4}=G^{4} (0)$. Then we have that
$$
c_{2} < 0 <c_{4} <c_{3} <c_{1}.
$$
Let $J=[c_{2}, c_{1}]$ and let $J_{0} = [c_{2}, c_{4}]$ and $J_{1}=[c_{3}, c_{1}]$.

\begin{figure}[htp]~\label{fig4}
\centering{
\includegraphics[scale=0.6]{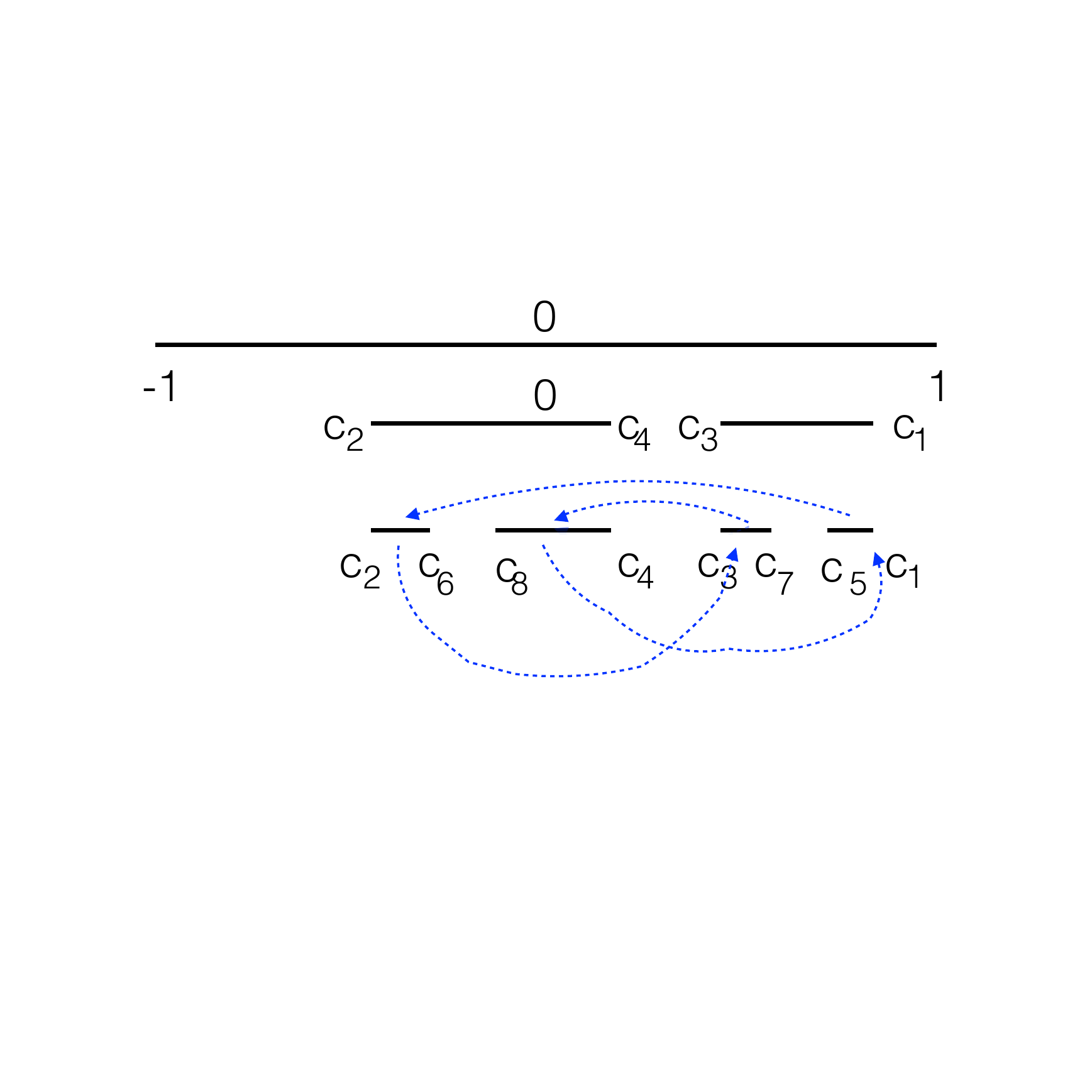}}
\vspace*{-100pt}
\caption{}
\end{figure}
We now define a $C^{3}$ expanding map $E(x)$. Let $\beta \in (0,1)$ be the fixed point of $G$. Define
$$
E(x) =\left\{ \begin{array}{ll}
          -\beta^{-1} x, &x\in J_{0};\cr
          -\beta^{-1} g(x), & x\in J_{1}
                   \end{array}\right.
$$
It has been shown in~\cite{JMS} that $E(x)$ is expanding, that is, $|E'(x)| >1$.
Therefore, the non-escaping set
$$
NE = \cap_{n=0}^{\infty} E^{-n} (J)
$$
is a Cantor set (see~\cite[Chapter 1]{JBook} for the detailed proof of this fact).

Let
$$
PC= \{ c_{n} =G^{n} (0)\}_{n=1}^{\infty}
$$
be the set of all post-critical points for $G$.
Then the attracting Cantor set for $G$ satisfies that
$$
\Lambda_{G} =\overline{PC}.
$$

Since ${\mathcal R}(G)=G$, we can see that endpoints of all intervals in $\cap_{n=0}^{N} E^{-n} (J)$ are contained in $PC$. This implies that
$$
NE = \Lambda_{G}.
$$

Since $E(x)$ is an expanding map of degree $2$, we can conjugate $E|NE$ to the shift map $\sigma$ of the symbolic space (the $2$-adic space in Section 5)
$$
{\mathbb Z}_{2} = \prod_{n=0}^{\infty} \{0, 1\}=\{ w=i_{0}i_{1}\cdots i_{n-2}i_{n-1}\cdots\;|\; i_{n-1} \in \{0, 1\}, \; n=1, 2, \cdots\}
$$
as follows.
Consider the initial partition
$$
\eta_{0} = \{ J_{0}, J_{1}\}.
$$
We can pull back it to get the $n^{th}$-partition
\begin{equation}~\label{part}
\eta_{n} = E^{-n} (\eta_{0}) =\{ J_{w_{n}}\}
\end{equation}
where $w_{n}=i_{0}i_{1}\cdots i_{n-2}i_{n-1}$.
We label intervals in $\eta_{n}$ inductively as that
$$
E(J_{w_{n}}) = J_{\sigma (w_{n})}\in \eta_{n-1},
$$
where $\sigma (w_{n}) = i_{1}\cdots i_{n-2}i_{n-1}$. From this labelling, we have that
$$
\cdots \subset J_{i_{0}i_{1}\cdots i_{n-2}i_{n-1}}\subset J_{i_{0} i_{1}\cdots i_{n-2}}\subset \cdots \subset J_{i_{0}}.
$$
Thus for any $w=i_{0}i_{1}\cdots i_{n-2}i_{n-1} \cdots \in {\mathbb Z}_{2}$, we have that
the set
$$
\cap_{n=1}^{\infty} J_{i_{0}i_{1}\cdots i_{n-1}} =\{ x_{w}\}
$$
contains one point. Set $\pi (w) =x_{w}$.

Suppose the Cantor set $NE$ is a metric space with the metric induced from the Lebesgue metric on $I$
and suppose ${\mathbb Z}_{2}$ is a metric space with the metric defined as
$$
d (w, w') = \frac{1}{2^{n}}
$$
if $w=i_{0}i_{1}\cdots i_{n-1} i_{n}\cdots$ and $w'=i_{0}i_{1}\cdots i_{n-1} i_{n}'\cdots$ and if $i_{n}\not= i_{n}'$.
Then we have that
$$
\pi : {\mathbb Z}_{2} \to NE
$$
is a homeomorphism such that
$$
\pi \circ E = \sigma \circ \pi.
$$

Now we define the add machining $add(w)=w+1$ on ${\mathbb Z}_{2}$ as follows.
For any $w=i_{0}i_{1}\cdots i_{n-1}\cdots$, if $i_{0}=0$, then $add (w) =1i_{1}\cdots i_{n-1}\cdots$ and if $i_{0}=1$, then the first digit of $add (w)$ is $0$ and consider $i_{1} +1$, etc. From the construction of $NE=\Lambda_{G}$, we have that
$$
G(x_{w}) =x_{add (w)}.
$$
This implies that
$$
G (x) = \pi \circ add\circ \pi^{-1} (x)
$$
and, furthermore,
$
G^{n} (x) =\pi \circ add^{n} \circ \pi^{-1} (x)
$
for $x\in NE$.
Since the adding machine $add$ is isometric then equicontinuous and $\pi$ is a conjugacy, we have that $\mathcal{G}=(NE, G)$ is equicontinuous.

Now for any $T\in {\mathcal F}_{\infty}$, we have a homeomorphism $h: I\to I$ such that
$$
T^{n} =h\circ G^{n}\circ h^{-1} \quad \hbox{on $I$}.
$$
Both $h|\Lambda_{G}$ and $h^{-1}|\Lambda_{T}$ are uniformly continuous.
So we have that $\mathcal{T}=(\Lambda_{T}, T)$ is equicontinuous.

Now what we have proved is that $T$ has countably many minimal subsets 
$K_{0}=\{-1\}$, $K_{1}=\{p\}$, $K_{n+1}=Orb(n)$ for all $n\geq 1$, and $K_{\infty}=\Lambda_{T}$.

Since $K_{0}$ or $K_{1}$ contains only one fixed point, it is easy to see that  
$$
\hbox{\rm Basin} (K_{0}) =\{-1,1\}, \quad \hbox{\rm Basin} (K_{1}) =\cup_{m=0}^{\infty} T^{-n} (K_{1}).
$$

Similarly, since $K_{n+1}$ contains only one periodic cycle of period $2^{n}$ for $n\geq 1$, 
$$
\hbox{\rm Basin} (K_{n+1}) =\cup_{m=0}^{\infty} T^{-n} (K_{n+1}), \quad n=1, 2, \cdots.
$$
Now we claim that 
$$
\hbox{\rm Basin} (K_{\infty}) =[-1,1]\setminus \cup_{n=0}^{\infty} \hbox{\rm Basin} (K_{n}).
$$

To prove this claim, we consider $G$ and the induced expanding map $E$. 
Let $\eta_{n}$ be the $n^{th}$-partition in (\ref{part}). Since $E$ is expanding,
we have a constant $0<\tau<1$ such that
$$
\tau_{n} = \max_{J\in \eta_{n}} |J| \leq \tau^{n}, \quad \forall n\geq 1.
$$ 
The map $G$ permutes intervals in $\eta_{n}$ (see Fig. 4).
For any 
$$
x\in [-1,1]\setminus \cup_{n=0}^{\infty} \hbox{\rm Basin}_{G} (K_{n}),
$$ 
we know that $G^{n}(x)\to NE=\Lambda_{G}$ as $n\to \infty$. Note that 
$$
\Lambda_{G}=NE=\cap_{n=1}^{\infty} \cup_{J\in \eta_{n}} J.
$$
Therefore, for any $\epsilon >0$, we have an integer $N'>0$ such that $\tau_{N'}\leq \tau^{N'}<\epsilon$. We also have an integer $N''\geq N'$ such that 
$G^{N''} (x)$ is in some interval $J\in \eta_{N'}$. Take a point $z_{0}\in J\cap \Lambda_{G}$. Then we have that  
$$
|G^{N''}(x) - z_{0}| <\epsilon.
$$
Since $G:\Lambda_{G}\to \Lambda_{G}$ is a homeomorphism, we have a point $z\in J'\cap \Lambda_{G}$ for $J'\in \eta_{N'}$ such that $z_{0}=G^{N''} (z)$ (refer to Fig. 4).
This implies that 
$$
|G^{N''}(x)-G^{N''} (z)|<\epsilon.
$$ 
Since $G$ permutes intervals in $\eta_{N}'$, for any $n\geq N''$, we have that $G^{n}(x)$ and $G^{n}(z)$ are in the same interval in $\eta_{N'}$, 
this further implies that 
$$
|G^{n}(x)-G^{n}(z)| <\epsilon, \quad \forall n\geq N''.
$$ 
Thus we have that
$$
   \limsup_{N\to\infty} \frac{1}{N} \sum_{n=1}^{N} |G^n (x)-G^n (z)| <\epsilon.
$$ 
This implies that 
$$
\hbox{\rm Basin}_{G} (K_{\infty}) =[-1,1]\setminus \cup_{n=0}^{\infty} \hbox{\rm Basin}_{G} (K_{n}).
$$
We proved the claim for $G$.
Since $T^{n} =h\circ G^{n}\circ h^{-1}$ and since both $h$ and $h^{-1}$ are uniformly continuous on $I$,  we proved also the claim for $T$.
 
Now we knew that $T: K_{\infty}\to K_{\infty}$ is minimal and equicontinuous. For each $n\geq 0$,
the map $T: K_{n}\to K_{n}$ is minimal and equicontinuous because it is just a periodic cycle. 
Thus we finally proved that ${\mathcal X}=(I,T)$ is {\bf MMA} and {\bf MMLS}.
 \end{proof}

One consequence of Theorem~\ref{main1} and Theorem~\ref{main2} is the following corollary with the M\"obius function as a special example. 

\medskip
\begin{corollary}~\label{scfeig}
	Any oscillating sequence is linearly disjoint from 
	 all Feigenbaum zero topological entropy flows.
 \end{corollary}

 \medskip
 \begin{remark} 
 	After this paper was completed, we learnt that Karagulyan \cite{Kar} has proved that the M\"obius function is linearly disjoint from  all continuous interval maps with zero topological entropy. It can be proved that any oscillating sequence is linear disjoint from all continuous interval maps with zero topological entropy and the special property of M\"{o}bius function is not needed in  Karagulyan's proof, but only the oscillating property.
 We expect that all continuous interval maps with zero topological entropy are {\bf MMA} and {\bf MMLS}.
 \end{remark} 

 \section{\bf Circle Homeomorphisms}

In this section, we discuss flows of zero topological entropy  on the unit circle which are {\bf MMA} and {\bf MMLS} but not equicontinuous on its minimal set.

Recall that we use $S^{1}=\{ z\in {\mathbb C} \;|\; |z|=1\}$ to denote the unit circle.
Suppose $T: S^1\to S^1$ is a continuous map. Then from~\cite{CMY, Va}, we know that if
${\mathcal X} =(S^{1}, T)$ is equicontinuous, then either
\begin{itemize}
\item[(1)] $\deg (T)=0$ and for every $x\in S^1$, $T^{n}x\to E$ as $n\to \infty$, where $E$ is the fixed point set of $T^{2}$,
\item[(2)] $\deg(T)=1$ and $T$ is conjugate to a rigid rotation $R_{\rho}(z) =e^{2\pi i \rho} z$ for some $0\leq \rho <1$, or
\item[(3)] $\deg(T)=-1$ and $T^{2}$ is the identity.
\end{itemize}
The above three cases are simple and Theorem~\ref{main1} applies.
Moreover, the flow defined by an orientation-preserving circle homeomorphism is equicontinuous if and only if it is conjugate to a rigid rotation.
Therefore, we know that all oscillating sequences are linearly disjoint from all orientation-preserving rational circle homeomorphisms and 
all orientation-preserving irrational circle homeomorphisms which are not Denjoy counter-example.  

We are now concerning that if Theorem~\ref{main1} applies to all Denjoy counter-examples of orientation-preserving irrational circle homeomorphisms.
A Denjoy counter-example is an orientation-preserving circle homeomorphism which is only semi-conjugate to a rigid rotation $R_{\rho}(z)=e^{2\pi i \rho}z$ for an irrational rotation number $0<\rho <1$. That is, we have a continuous onto (but not 1-1) map $h: S^1\to S^1$ such that
\begin{equation}~\label{conj}
h\circ T =R_{\rho}\circ h
\end{equation}
on $S^1$.
Thus  ${\mathcal X} =(S^{1}, T)$ is not equicontinuous according to~\cite{CMY, Va}. Let us first state the Poincar\'e theorem on 
orientation-preserving circle homeomorphisms with irrational rotation numbers.

\medskip
\begin{theorem}[Poincar\'e Theorem]~\label{pt}
Suppose $T: S^1\to S^1$ is an orientation-preserving circle homeomorphism with irrational rotation number $\rho$.
Then there is a continuous onto map $h: S^1\to S^1$ satisfying (\ref{conj}).
Moreover, the $\omega$-limit set $\omega (z)$ for a point $z$ in the circle is either a Cantor set $\Lambda$, which is independent of $z$, or the whole circle $S^{1}$.
In the case $\omega (z)=S^{1}$, $h$ is a homeomorphism. In the case $\omega (z)=\Lambda$, a Cantor set, $S^{1}\setminus \Lambda =\cup J$ is the union of all wandering intervals.
Here $\Lambda$ is also called the non-wandering set for $T$. 
\end{theorem}

Denjoy constucted an example $T$ of an orientation-preserving circle homeomorphism with irrational rotation number such that for any $z \in S^{1}$, 
$\omega (z)=\Lambda$ is a Cantor set. In this case, $T: \Lambda \to \Lambda$ is minimal. We call such a map a Denjoy counter-example. 

Associated to a Denjoy counter-example, there are two flows ${\mathcal X}=(S^{1}, T)$ and ${\mathcal A}=(\Lambda, T)$. We have that the following theorem.

\medskip
\begin{theorem}~\label{ceq}
Suppose $T: S^{1}\to S^{1}$ is a Denjoy counter-example. The flow $\mathcal{X}= (S^{1}, T)$ is {\bf MMA} but the flow $\mathcal{A} =(\Lambda, T)$ is not equicontinuous.
However, the flow $\mathcal{A} =(\Lambda, T)$ is {\bf MLS}, thus, the flow $\mathcal{X}= (S^{1}, T)$ is {\bf MMA} and {\bf MMLS}.
 \end{theorem}

\begin{proof}
For any $x, y\in \Lambda$ close to each other, let $[x,y]$ be the shortest arc with endpoints $[x, y]$ in the circle. The arc length of $[x,y]$ is $|x-y|$. Assume the total length of $S^{1}$ is $1$. 

Since $T$ is a Denjoy counter-example,  $\Lambda$ is a Cantor set and 
$$
S^{1}\setminus \Lambda=\cup I
$$ 
is the union of all wondering intervals. 
We have that 
$$
\sum_{I} |I| \leq 1.
$$

For any $x\in S^{1}\setminus \Lambda$, it is in a wandering interval $I=(a,b)$, where $a, b\in \Lambda$. Let $I_{n} =T^{n}(I)$, $n=0, 1, \cdots$. Then we have that 
$$
\sum_{n=0}^{\infty} |I_{n}| \leq 1.
$$
Thus $|I_{n}|\to 0$ as $n\to \infty$.
Since $T^{n}(x) \in I_n$ and $T^{n}(a)\in \overline{I}_{n}$, we have that
$$
\lim_{n\to +\infty} |T^{n} (a)- T^{n}(x)| = \lim_{n\to +\infty}
|I_n| = 0,  
$$
which implies that $\mathcal{X}=(S^{1}, T)$ is {\bf MMA}.

Consider the set ${\mathcal W}$ of all wandering intervals, which  is countable. We partition it into  full orbits 
$$
{\mathcal W} =\{ I_{m,n}\; |\; I_{m, n} =T^{n} (I_{m, 0}),\;\; 0\leq m< L,\; n\in {\mathbb Z}\},
$$  
where $L$ is the number of full orbits of wandering intervals. Here
$L$ is either a positive integer or $\infty$. Since $\Lambda$ is a Cantor set, we have
$$
S^{1} =\overline{\bigcup_{0\le m < L} \bigcup_{n\in {\mathbb Z}} I_{m,n}}
$$ 

Let $h$ be the semi-conjugacy between $T$ and the rigid rotation $R_{\rho}$. See  (\ref{conj}). 
Then the image $ h(I_{m,n})$ is a singleton $\{x_{m,n}\}$
 and we write $x_{m, n}=h(I_{m, n})$. By (\ref{conj}),
\begin{equation}\label{hh1}
\{ x_{m,n} =h(I_{m, n}) = R^{n}_{\rho} (x_{m,0}) \;|\; n\in {\mathbb Z}\}
\end{equation}  
is the full orbit of $x_{m, 0}$ under iterations of $R_{\rho}$ for any fixed $0\leq m<L$.

Also notice that  $h$ is  monotone increasing. So,  if $x$ and $y$ are two points in the Cantor set $\Lambda$， we have
\begin{equation}\label{hh2}
	 I_{m,n} \subset [x, y]     \quad\mbox{\rm  iff}\quad   x_{m,n} \in [h(x), h(y)]. 
\end{equation}      

The rigid rotation $R_{\rho}$ is uniquely ergodic 
with the Lebesgue measure as the unique invariant measure.
This fact, together with (\ref{hh1}),
 implies that for any $z, w\in S^{1}$, 
\begin{equation}\label{hh3}
\lim_{n\to \infty} \frac{\#(\{ k\;|\; x_{m, k}\in [z, w], \; -n \leq k\leq 0\})}{n} = |z-w|.
\end{equation}
This allows us to prove that ${\mathcal A}=(\Lambda, T)$ is not equicontinuous. In fact, 
given any $x\not= y\in \Lambda$, by (\ref{hh2}) and (\ref{hh3})
we have that 
\begin{eqnarray*}
& & \lim_{n\to \infty} \frac{\#(\{ k\;|\; I_{0, k}\subset [x, y], \; -n \leq k\leq 0\})}{n} \\
&=&
\lim_{n\to \infty} \frac{\#(\{ k\;|\; x_{0,k}\in [h(x), h(y)], \; -n \leq k\leq 0\})}{n} = |h(x)-h(y)| >0.
\end{eqnarray*}
Thus there are infinitely many  negative integers $k$ such that $I_{0,k}\subset [x,y]$. For such $k$'s we have 
$$
I_{0,0} =T^{|k|} (I_{0, k}) \subset [T^{|k|}(x), T^{|k|}(y)]
$$
so that 
$$
|T^{|k|}(x)-T^{|k|}(y)|\geq |I_{0}|.
$$
for infinitely many positive integers $|k|$. Therefore,  ${\mathcal A} =(\Lambda, T)$ is not equicontinuous.

Next we prove that ${\mathcal A} =(\Lambda, T)$ is {\bf MLS}.
We first assume that the non-wandering set $\Lambda$ has zero Lebesgue measure.
Under this assumption, we have that 
$$
\sum_{0\le m< L} \sum_{n\in {\mathbb Z}} |I_{m,n}| =1.
$$
Moreover, for any two points $x\not= y\in \Lambda$,
\begin{equation}~\label{zeq}
|x-y| =\sum_{I\in {\mathcal W}, I\subset [x,y]} |I|.
\end{equation}

For any $N\ge 1$, consider the rectangle
$$
      R_N=\{ (m, n)\in \mathbb{N}\times \mathbb{Z}\; |\;  0\le m \le N \wedge (L-1), |n|\le N\},
$$
where $N\wedge (L-1)= min\{ N, L-1\}$.
For any $\epsilon >0$, we have a large integer $N_{0}\ge 1$ such that 
$$
\sum_{(m, n )\in \mathbb{N}\times \mathbb{Z} \setminus R_{N_0}} |I_{m,n}| < \epsilon.
$$
Since $h$ is uniformly continuous, we can find a $0< \delta\leq \epsilon$ such that for any $x\not= y\in \Lambda$ with $|x-y|<\delta$,
$$
0< |h(x)-h(y)| < \frac{\epsilon}{(N_{0}\wedge (L-1)+1)(2N_{0}+1)}.
$$ 
Let 
$$
E=\bigcup_{(m, l)\in R_{N_0}} E_{m,l}
$$
where
$$
E_{m,l}=\{0\leq k<\infty \;|\; I_{m,-k+l} \subset [x,y]\}.
$$
For any fixed $(m, l)\in R_{N_0}$,  we have
\begin{eqnarray*}
	\lim_{n\to \infty} \frac{\#(E_{m,l}\cap [1,n])}{n} 
	& = &
        \lim_{n\to \infty} \frac{\#(\{ k\;|\; I_{m, k+l}\subset [x, y], \; -n \leq k\leq 0\})}{n} \\
&= &  
        \lim_{n\to \infty} \frac{\#(\{ k\;|\; x_{m,k+l}\in [h(x), h(y)], \; -n \leq k\leq 0\})}{n} 
\\
& = &  |h(x)-h(y)| >0.
\end{eqnarray*}
Therefore
the  density of $E$ exists and  is equal to  
$$
\overline{D} (E)= \limsup_{n\to \infty} \frac{\#(E\cap [1,n])}{n} = (N_{0}\wedge(L-1)+1) (2N_{0}+1) |h(x)-h(y)| < \epsilon
$$
For any $n\not\in E$, the interval $[T^{n}(x), T^{n}(y)]$ contains only wandering intervals from
$$
{\mathcal W}_{N_{0}} =\{ I_{m,n} =T^{n} (I_{m, 0})\; |\;  (m,n)\in \mathbb{N}\times \mathbb{Z} \setminus R_{N_0}\}.
$$ 
From (\ref{zeq}),  we have that
$$
|T^{n}(x)-T^{n}(y)| \leq \sum_{(m,n) \in \mathbb{N}\times \mathbb{Z} \setminus R_{N_0}} |I_{m,n}| <\epsilon.
$$
This implies that ${\mathcal A}=(\Lambda, T)$ is {\bf MLS}. Thus it is mean-equicontinuous.
Since the flow ${\mathcal X} =(S^{1}, T)$ has only one mean minimal attractor $\Lambda$, so it is {\bf MMLS}. 

When the non-wandering set $\Lambda$ has a positive Lebesgue measure, there is another Denjoy counter-example $\widetilde{T}$ whose non-wandering set $\widetilde{\Lambda}$ has zero Lebesgue measure such that both of them are semi-conjugate to the same rigidity rotation $R_{\rho}$. That is, we have two continuous onto maps $h, \widetilde{h}: {\mathbb T}\to {\mathbb T}$ such that 
$$
h\circ T= R_{\rho}\circ h \quad \hbox{and}\quad \widetilde{h}\circ \widetilde{T}= R_{\rho} \circ \widetilde{h}.
$$
Now we can define
$$
\widehat{h}  : {\mathbb T}\to {\mathbb T}
$$  
which formally can be written as $\widetilde{h}^{-1} \circ h$. It maps every wandering interval to the corresponding wandering interval bijectively and preserves orders of wandering intervals. And it maps $\Lambda$ to $\widetilde{\Lambda}$ bijectively. Thus it is a homeomorphism of the circle such that
$$
\widehat{h} \circ T = \widetilde{T} \circ \widehat{h}.
$$ 
Since $\widetilde{T}$ is {\bf MMLS} and since $T$ is topologically conjugate to $\widetilde{T}$, we get that $T$ is {\bf MMLS}.
\end{proof}

One consequence of Theorem~\ref{main1} and Theorem~\ref{main2} and Theorem~\ref{ceq} is the following corollary with the M\"obius function as a special example. 

\medskip
\begin{corollary}~\label{scdenjoy}
Any oscillating sequence satisfying the growth condition (\ref{growth}) is linearly disjoint from any Denjoy counter-example. 
\end{corollary}

In particular, we have 
\medskip
\begin{corollary}[Karagulyan~\cite{Kar}]~\label{Kar}
The M\"obius function is linearly disjoint  from all orientation-preserving circle homeomorphisms.  
\end{corollary}

 \end{document}